\documentclass[twoside,11pt,reqno]{amsart}
\usepackage{amsmath,amssymb,amscd,mathrsfs,epic,wasysym,latexsym,tikz,mathrsfs,cite,hyperref}
\usepackage{pb-diagram}
\usepackage[matrix,arrow]{xy}

\makeatletter

\numberwithin{equation}{section}

\newtheorem{Proposition}[equation]{Proposition}
\newtheorem{Lemma}[equation]{Lemma}
\newtheorem{Theorem}[equation]{Theorem}
\newtheorem{Corollary}[equation]{Corollary}
\theoremstyle{definition}  

\newtheorem{Remark}[equation]{Remark}

\newtheorem{Problem}[equation]{Problem}
\newtheorem{Hypothesis}[equation]{Hypothesis}

\let\<\langle
\let\>\rangle

\newcommand\Comment[2][\relax]{\space\par\medskip\noindent%
   \fbox{\begin{minipage}{\textwidth}\textbf{Comment\ifx\relax#1\else---#1\fi}\newline%
        #2\end{minipage}}\medskip
}

\def\bi{\text{\boldmath$i$}}
\def\bj{\text{\boldmath$j$}}

\def\b1{\text{\boldmath$1$}}

\def\pmod#1{\text{ }(\text{\rm mod } #1)\,}

\newcommand{\Hom}{\operatorname{Hom}}
\newcommand{\Ext}{\operatorname{Ext}}
\newcommand{\EXT}{\operatorname{Ext}}
\newcommand{\ext}{\operatorname{ext}}
\newcommand{\Tor}{\operatorname{Tor}}
\newcommand{\End}{\operatorname{End}}

\newcommand{\id}{\operatorname{id}}

\newcommand{\cha}{\operatorname{char}}

\newcommand{\Z}{\mathbb{Z}}

\def\eps{{\varepsilon}}
\def\phi{{\varphi}}

\newcommand{\catC}{{\mathbf C}}

\newcommand{\Torsion}{{\tt Tors}}

\newcommand{\ga}{\gamma}

\newcommand{\la}{\lambda}
\newcommand{\La}{\Lambda}
\newcommand{\al}{\alpha}
\newcommand{\be}{\beta}

\newcommand{\de}{\delta}
\newcommand{\De}{\Delta}

\def\id{\mathop{\mathrm {id}}\nolimits}

\newcommand{\Ind}{{\mathrm {Ind}}}

\def\rank{\mathop{\mathrm{ rank}}\nolimits}

\newcommand{\Res}{{\mathrm {Res}}}

\newcommand{\Q}{{\mathbb Q}}

\renewcommand{\mod}{\bmod \,}

\def\b{\mathfrak{b}}
\def\k{\Bbbk}
\def\x{x}

\def\height{{\operatorname{ht}}}

\def\op{{\mathrm{op}}}

\def\mod#1{#1\!\operatorname{-mod}}

\def\HOM{\operatorname{Hom}}

\def\CH{{\operatorname{ch}_q\,}}
\def\DIM{{\operatorname{dim}_q\,}}
\def\RANK{{\operatorname{rank}_q\,}}

\def\words{{ I}}

\def\KP{\operatorname{KP}}

{\catcode`\|=\active
  \gdef\set#1{\mathinner{\lbrace\,{\mathcode`\|"8000%
  \let|\midvert #1}\,\rbrace}}
}
\def\midvert{\egroup\mid\bgroup}

\colorlet{darkgreen}{green!50!black}
\tikzset{dots/.style={very thick,loosely dotted},
         greendot/.style={fill,circle,color=darkgreen,inner sep=1.5pt,outer sep=0}
}
\def\greendot(#1,#2){\node[greendot] at(#1,#2){}}

\newenvironment{braid}{
  \begin{tikzpicture}[baseline=6mm,blue,line width=1pt, scale=0.4,
                      draw/.append style={rounded corners},
                      every node/.append style={font=\fontsize{5}{5}\selectfont}]%
  }{\end{tikzpicture}
}

\def\Grid(#1,#2){
  \draw[very thin,gray,step=2mm] (0,0)grid(#1,#2);
  \draw[very thin,darkgreen,step=10mm] (0,0)grid(#1,#2);
}

\newcommand\Tableau[2][\relax]{
  \begin{tikzpicture}[scale=0.5,draw/.append style={thick,black}]
    \ifx\relax#1\relax%
    \else 
      \foreach\box in {#1} { \filldraw[blue!30]\box+(-.5,-.5)rectangle++(.5,.5); }
    \fi
    \newcount\row\newcount\col
    \row=0
    \foreach \Row in {#2} {
       \col=1
       \foreach\k in \Row {
          \draw(\the\col,\the\row)+(-.5,-.5)rectangle++(.5,.5);
          \draw(\the\col,\the\row)node{\k};
          \global\advance\col by 1
       }
       \global\advance\row by -1
    }
  \end{tikzpicture}
}

\newcommand\YoungDiagram[2][\relax]{
  \begin{tikzpicture}[scale=0.5,draw/.append style={thick,black}]
    \ifx\relax#1\relax%
    \else 
    \foreach\box in {#1} {
      \filldraw[blue!30]\box rectangle ++(1,1);
    }
    \fi
    \newcount\row
    \row=0
    \foreach \col in {#2} {
       \draw(1,\the\row)grid ++(\col,1);
       \global\advance\row by -1
    }
  \end{tikzpicture}
}

\begin{document}
\title[Homomorphisms between standard modules]{{\bf Homomorphisms between standard modules over finite type KLR algebras}}

\author{\sc Alexander S. Kleshchev}
\address{Department of Mathematics\\ University of Oregon\\
Eugene\\ OR~97403, USA}
\email{klesh@uoregon.edu}

\author{\sc David J. Steinberg}
\address{Department of Mathematics\\ University of Oregon\\
Eugene\\ OR~97403, USA}
\email{dsteinbe@uoregon.edu}

\subjclass[2010]{16G99, 16E05, 17B37}

\thanks{Research supported by the NSF grant DMS-1161094.}

\begin{abstract}
Khovanov-Lauda-Rouquier algebras of finite Lie type come with  families of {\em standard modules}, which under the Khovanov-Lauda-Rouquier categorification correspond to PBW-bases of the positive part of the corresponding quantized enveloping algebra. We show that there are no non-zero homomorphisms between distinct standard modules and all non-zero endomorphisms of a standard module are injective. We obtain  applications to the extensions between standard modules and modular representation theory of KLR algebras.
\end{abstract}

\maketitle

\section{Introduction}

 Khovanov-Lauda-Rouquier algebras of finite Lie type possess {\em affine quasi-hereditary structures} \cite{BKM, Kato, KLM, KlL, KX, Kdonkin}. In particular, they come with important families of modules which are called {\em standard}. Under the Khovanov-Lauda-Rouquier categorification \cite{KL1,Ro}, standard modules correspond to PBW-monomials in the positive part of the corresponding quantized enveloping algebra, see \cite{BKM, Kato}. 

Affine quasihereditary structures are parametrized by {\em convex orders} on the sets of positive roots of the corresponding root systems. In this paper we work with an arbitrary convex order and an arbitrary finite Lie type. When working with the KLR algebra $H_\al$ for any $\al\in Q^+$, the standard modules $\De(\la)$ are labeled by $\la\in\KP(\al)$, where $\KP(\al)$ is the set of Kostant partitions of $\al$. 
With these conventions, our main result is as follows:

\vspace{2mm}
\noindent
{\bf Theorem A.}
{\em
Let $\alpha\in Q^+$ and  $\lambda,\mu\in\KP(\al)$.  If $\lambda \neq \mu$, then 
$$
\Hom_{H_\alpha}(\Delta(\lambda), \Delta(\mu)) = 0.
$$
}

When $\la\not\leq\mu$, it is clear that $\Hom_{H_\alpha}(\Delta(\lambda), \Delta(\mu)) = 0$, but for $\la<\mu$, we found this fact surprising. Theorem A is proved in Section~\ref{SThA}.

The case $\la=\mu$ is also well-understood. In fact, the endomorphism algebras of the standard modules are naturally isomorphic to certain algebras of symmetric functions, see Theorem~\ref{TInjGeneral}. Now, Theorem A can be complemented by the following (folklore) observation  and compared to the main result of \cite{BCGM}:

\vspace{2mm}
\noindent
{\bf Theorem B.}
{\em
Let $\alpha\in Q^+$ and  $\lambda\in\KP(\al)$. Then every non-zero $H_\al$-endomorphism of $\De(\la)$ is injective. 
}

\vspace{2mm}
For reader's convenience, we prove Theorem B in Section~\ref{SSEnd}. 

Theorem A turns out to have some applications to modular representation theory of KLR algebras, which are pursued in Section~\ref{SRedModP}. Note that KLR algebras are defined over an arbitrary ground ring $k$, and when we wish to emphasize this fact, we use the notation $H_{\al,k}$. Using the $p$-modular system $(F,R,K)$ with $F=\Z/p\Z$, $R=\Z_p$ and $K=\Q_p$, we can reduce modulo $p$ any irreducible $H_{\al,K}$-module. An important problem is to determine when these reductions remain irreducible, see \cite{KRbz,Williamson}. This problem can be reduced to homological questions involving standard modules.

In Section~\ref{SRedModP}, we show that standard modules have universal $R$-forms $\De(\la)_R$ such that $\De(\la)_R\otimes_Rk\cong \De(\la)_k$ for any field $k$. Then an application of the Universal Coefficient Theorem and Theorem A yields: 

\vspace{2mm}
\noindent
{\bf Theorem C.}
{\em
Let $\al\in Q^+$ and $\la,\mu\in\KP(\al)$. Then the $R$-module 
$$\Ext^1_{H_{\al,R}}(\De(\la)_R,\De(\mu)_R)$$ is torsion-free.  Moreover, $$\DIM \Ext^1_{H_{\al,F}}(\De(\la)_F,\De(\mu)_F)=\DIM\Ext^1_{H_{\al,K}}(\De(\la)_K,\De(\mu)_K)$$ if and only if $\Ext^2_{H_{\al,R}}(\De(\la)_R,\De(\mu)_R)$ is torsion-free. 
}
\vspace{2mm}

As a final application, using a universal extension procedure, we construct $R$-forms $Q(\la)_R$ of the projective indecomposable modules $P(\la)_K$, and prove: 

\vspace{2mm}
\noindent
{\bf Theorem D.}
{\em
Let $\al\in Q^+$. Then reductions modulo $p$ of all irreducible $H_{\al,K}$-modules are irreducible if and only if one of the following equivalent conditions holds:
\begin{enumerate}
\item[{\rm (i)}] $Q(\la)_R\otimes_RF$ is a projective $H_{\al,F}$-module for all $\la\in\KP(\al)$;
\item[{\rm (ii)}] $\Ext^1_{H_{\al,F}}(Q(\la)_R\otimes_RF,\De(\mu)_F)=0$ for all $\la,\mu\in\KP(\al)$;
\item[{\rm (iii)}] $\Ext^2_{H_{\al,R}}(Q(\la)_R,\De(\mu)_R)$ is torsion-free for all $\la,\mu\in\KP(\al)$. 
\end{enumerate}

}
\vspace{2mm}


\section{Preliminaries}\label{SPrel}

\subsection{KLR algebras}
We follow closely the set up of \cite{BKM}. In particular, $R$ is an irreducible root system with simple roots $\{\alpha_i \mid i \in I\}$ and the corresponding set of positive roots $R^+$. Denote by $Q$ the root lattice and by $Q^+\subset Q$ the set of $\mathbb{Z}_{\geq 0}$-linear combinations of simple roots, and write $\height(\alpha) = \sum_{i \in I} c_i$ for $\alpha = \sum_{i \in I} c_i \alpha_i\in Q^+$. The standard symmetric bilinear form $Q\times Q\to\Z,\ (\al,\be)\mapsto \al\cdot\be$ is normalized so that $d_i:=(\al_i\cdot\al_i)/2$ is equal to $1$ for the short simple roots $\al_i$. We also set $d_\be:=(\be\cdot\be)/2$ for all $\be\in R^+$. 
The Cartan matrix is $C = (c_{i,j})_{i,j \in I}$ with $c_{i,j} := (\alpha_i \cdot \alpha_j) / d_i$.

Fix a commutative unital ring $k$ and an element $\alpha \in Q^+$  of height $n$. The symmetric group $S_n$ with simple transpositions $s_r:=(r\:r\!+\!1)$ acts on the set $\words^\alpha = \{ \bi = i_1 \cdots i_n \in I^n \mid  \sum_{j=1}^n \alpha_{i_j} = \alpha \}$ by place permutations. Choose signs $\epsilon_{i,j}$ for all $i,j \in I$ with $c_{ij}<0$ so that $\epsilon_{i,j}\epsilon_{j,i}=-1$. With this data, Khovanov-Lauda \cite{KL1,KL2} and Rouquier \cite{Ro} define the graded $k$-algebra $H_\al$ with unit $1_\al$, called the {\em KLR algebra}, given by generators 
$$\{1_\bi \mid  \bi \in \words^\alpha \} \cup \{ x_1, \dots, x_n\} \cup \{\tau_1, \dots, \tau_{n-1}\}$$
subject only to the following relations
\begin{itemize}
\item$\x_r \x_s = \x_s \x_r$;
\item $1_\bi1_\bj=\de_{\bi,\bj}1_\bi$ and $\sum_{\bi\in\words^\al}1_\bi=1_\al$; 
\item
$\x_r 1_\bi = 1_\bi\x_r$ and 
$\tau_r 1_\bi = 1_{s_r\cdot \bi}
\tau_r$;
\item 
$(x_t \tau_r-\tau_r x_{s_r(t)})1_\bi  
= \de_{i_r,i_{r+1}}(\de_{t,r+1}-\de_{t,r})1_\bi;$
\item
$
\tau_r^2 1_\bi = 
\left\{
\begin{array}{ll}
0&\text{if $i_r=i_{r+1}$,}\\
\eps_{i_r,i_{r+1}}\big({x_r}^{-c_{i_r,i_{r+1}}}-{x_{r+1}}^{-c_{i_{r+1},i_r}}\big)1_\bi&\text{if
  $c_{i_r,i_{r+1}}< 0$,}\\
1_\bi&\text{otherwise;}
\end{array}\right.
$
\item $\tau_r \tau_s = \tau_s \tau_r$ if $|r-s|>1$;
\item 
$(\tau_{r+1} \tau_{r} \tau_{r+1} -
  \tau_{r}\tau_{r+1}\tau_{r}) 1_\bi
=$\\
\phantom{\hspace{15mm}}$\left\{
\begin{array}{ll}
\displaystyle\sum_{r+s=-1-c_{i_r,i_{r+1}}}
\!\!\!\!\!\!\!
\eps_{i_r,i_{r+1}} x_r^r x_{r+2}^s 1_\bi&\text{if
  $c_{i_r,i_{r+1}} < 0$ and $i_r =
  i_{r+2}$,}\\
\hspace{9mm}0&\text{otherwise.}
\end{array}\right.
$
\end{itemize}
The KLR algebra is graded with $\deg 1_{\bi}=0$, $\deg (x_r1_{\bi})= 2d_{i_r}$ and $\deg (\tau_{r}1_{\bi})=-\alpha_{i_r} \cdot \alpha_{i_{r+1}}$.  

For each element $w \in S_n$, fix a reduced expression $w=s_{r_1}\cdots s_{r_l}$ which determines an element $\tau_w = \tau_{r_1} \cdots \tau_{r_l}\in H_\al$.

\begin{Theorem} \label{TBasis} 
 {\bf (Basis Theorem)} \cite[Theorem 2.5]{KL1} 
The sets 
\begin{equation}\label{EBasis}
\{\tau_w x_1^{a_1}\cdots x_n^{a_n} 1_{\bi}\}\quad\text{and}\quad
\{x_1^{a_1}\cdots x_n^{a_n} \tau_w  1_{\bi}\},
\end{equation}
with $w$ running over $S_n$, $a_r$ running over $\Z_{\geq 0}$, and $\bi$ running over $\words^\alpha$, 
 are $k$-bases for $H_\alpha.$ 
\end{Theorem} 

It follows that $H_\al$ is Noetherian if $k$ is Noetherian. It also follows that for any $1\leq r\leq n$, the subalgebra $k[x_r]\subseteq H_\al$, generated by $x_r$, is isomorphic to the polynomial algebra $k[x]$---this fact will be often used without further comment.  Moreover, for each $\bi \in I^\al$, the subalgebra $\mathcal{P}(\bi) \subseteq 1_\bi R_\al 1_\bi$ generated by $\{x_r1_\bi \mid 1 \leq r \leq n\}$ is isomorphic to a polynomial algebra in $n$ variables.  By defining $\mathcal{P}:= \bigoplus_{\bi \in I^\al} \mathcal{P}(\bi)$, we obtain a linear action of $S_{n}$ on $\mathcal{P}$ given by 
$$w \cdot x_1^{a_1} \cdots x_n^{a_n}1_\bi = x_{w(1)}^{a_1} \cdots x_{w(n)}^{a_n} 1_{w \cdot \bi},$$
for any $w \in S_n$, $\bi \in I^\al$ and $a_1, \dots, a_n \in \Z_{\geq 0}$.  Setting $\Lambda(\al):= \mathcal{P}^{S_n},$ we have:

\begin{Theorem} \label{TCenter}
\cite[Theorem 2.9]{KL1}
$\Lambda(\al)$ is the center of $H_\al$.
\end{Theorem}

If $H$ is a Noetherian graded $k$-algebra, we denote by $\mod{H}$ the category of finitely generated graded left $H$-modules. The morphisms in this category are all homogeneous degree zero  $H$-homomorphisms, which we denote $\hom_{H}(-,-)$. 
For  $V\in\mod{H}$, let $q^d V$ denote its grading shift by $d$,  so if $V_m$ is the degree $m$ component of $V$, then $(q^dV)_m= V_{m-d}.$ More generally, for a Laurent polynomial $a=a(q)=\sum_{d}a_dq^d\in\Z[q,q^{-1}]$ with non-negative coefficients, we set $a V:=\bigoplus_d(q^d V)^{\oplus a_d}$.

For $U,V\in \mod{H}$,
we set 
$\HOM_H(U, V):=\bigoplus_{d \in \Z} \HOM_H(U, V)_d$, 
where
$$
\HOM_H(U, V)_d := \hom_H(q^d U, V) = \hom_H(U, q^{-d}V).
$$
We define $\operatorname{ext}^m_H(U,V)$ and
$\EXT^m_H(U,V)$ similarly. 
Since $U$ is finitely generated, $\HOM_H(U,V)$ can be identified in the obvious way with the set of all $H$-module homomorphisms  ignoring the gradings. A similar result holds for $\EXT^m_H(U,V)$, since $U$ has a resolution by finitely generated projective modules. We use $\cong$ to denote an isomorphism in $\mod{H}$  and $\simeq$ 
an isomorphism up to a degree shift.

We always work in the category $\mod{H_\al}$, in particular all $H_\al$-modules are assumed to be finitely generated graded. 
Also, until Section~\ref{SRed}, we assume that $k$ is a field. Let $q$ be a variable, and $\Z((q))$ be the ring of Laurent series. The quantum integers $[n]=(q^n-q^{-n})/(q-q^{-1})$ and expressions like $1/(1-q^2)$ are always interpreted as elements of $\Z((q))$. Note that the {\em graded dimension} $\DIM 1_\bi H_\al$ is always an element of $\Z((q))$. So for any $V\in\mod{H_\al}$, its {\em formal character} $\CH V:=\sum_{\bi\in \words^\al}(\DIM 1_\bi V)\cdot\bi$ is an element of $\bigoplus_{\bi\in\words^\al}\Z((q))\cdot \bi$. Note also that $\CH(q^dV)=q^d\CH(V)$, where the first $q^d$ means the degree shift as introduced in the previous paragraph. We refer to $1_\bi V$ as the {\em $\bi$-weight space} of $V$ and to its vectors as {\em vectors of weight} $\bi$.

There is an anti-automorphism $\iota:H_\alpha \rightarrow H_\alpha$
which fixes all the generators. Given $V\in\mod{H_\al}$, we denote $V^\circledast := \HOM_{k}(V, k)$ viewed as a left $H_\al$-module 
via $\iota$. Note that in general $V^\circledast$ is not finitely generated as an $H_\al$-module, but we will apply $\circledast$ only to finite dimensional modules. In that case, we have $\CH V^{\circledast}=\overline{\CH V}$, where the bar means the {\em bar-involution}, i.e. the automorphism of $\Z[q,q^{-1}]$ that swaps $q$ and $q^{-1}$ extended to $\bigoplus_{\bi\in \words^\al}\Z[q,q^{-1}]\cdot \bi$. 

Let $\be_1,\dots,\be_m\in Q^+$ and $\al=\be_1+\dots+\be_m$. 
Consider the set of concatenations  $$I^{\be_1,\dots,\be_m}:=\{\bi^1\cdots\bi^m \mid \bi^1\in\words^{\be_1},\dots,\bi^m\in\words^{\be_m}\}\subseteq \words^\al.$$ 
There is a natural (non-unital) algebra embedding $H_{\be_1}\otimes\dots\otimes H_{\be_m}\to H_\al$, which sends the unit $1_{\be_1}\otimes\dots \otimes 1_{\be_m}$ to the idempotent 
\begin{equation}\label{EId}
1_{\be_1,\dots,\be_m}:=\sum_{\bi\in \words^{\be_1,\dots,\be_m}}1_\bi \in H_\al.
\end{equation}
We have an exact induction functor 
$$\Ind_{\be_1,\dots,\be_m}^{\al}=H_\al1_{\be_1,\dots,\be_m}\otimes_{H_{\be_1}\otimes\dots\otimes H_{\be_m}}-:\mod{(H_{\be_1}\otimes\dots\otimes H_{\be_m})}\to\mod{H_{\al}}.
$$ 
For $V_1\in\mod{H_{\be_1}}, \dots, V_m\in\mod{H_{\be_m}}$, we denote 
$$V_1\circ\dots\circ V_m:=\Ind_{\be_1,\dots,\be_m}^{\al} V_1\boxtimes \dots\boxtimes V_m.$$

\subsection{Standard modules}
The KLR algebras $H_\al$ are known to be \emph{affine quasihereditary} in the sense of \cite{Kdonkin}, see \cite{Kato, BKM, KlL}.  Central to this theory is the notion of {\em standard modules}, whose definition depends on a choice of a certain partial order. We first fix a \emph{convex order} on $R^+$, i.e. a total order such that whenever $\gamma$, $\beta$, and $\gamma + \beta$ all belong to $R^+$, $\gamma \leq \beta$ implies $\gamma \leq \gamma + \beta \leq \beta$. A \emph{Kostant} partition of $\alpha \in Q^+$ is a tuple $\la =(\la_1, \dots, \la_r)$ of positive roots with $\la_1 \geq \la_2 \geq \cdots \geq \la_r$ such that  $\la_1 + \cdots +\la_r = \alpha$.  Let $\KP(\alpha)$ denote the set of all Kostant partitions of $\alpha$ and for $\la$ as above define $\la_m'=\la_{r-m+1}$.  Now, we have a {\em bilexicographical partial order} on $\KP(\alpha)$, also denoted by $\leq$, i.e. if $\la =(\la_1, \dots, \la_r),\mu=(\mu_1, \dots, \mu_s)\in\KP(\al)$ then $\la < \mu$ if and only if the following two conditions are satisfied:
\begin{itemize}
\item $\la_1=\mu_1, \dots, \la_{l-1}=\mu_{l-1}$ and $\la_l < \mu_l$ for some $l$;
\item $\la'_1 =\mu'_1, \dots, \la_{m-1}' = \mu_{m-1}'$ and $\la_m' > \mu_m'$ for some $m$.
\end{itemize}

To every $\la \in \KP(\alpha)$, McNamara \cite{McN} (cf. \cite[Theorem 7.2]{KRbz}) associates an absolutely irreducible finite dimensional $\circledast$-self-dual $H_\alpha$-module $L(\lambda)$ so that $\{L(\la)\mid\la\in\KP(\al)\}$ is a complete irredundant set of irreducible $H_\al$-modules, up to isomorphism and degree shift. Since $L(\la)$ is $\circledast$-self-dual, its formal character is bar-invariant. 
The key special case is where $\la=(\al)$ for $\al\in R^+$, in which case $L(\la)=L(\al)$ is called a {\em cuspidal irreducible module}. 
For $m\in\Z_{> 0}$, we write $(\al^m)$ for the Kostant partition $(\al,\dots,\al)\in\KP(m\al)$, where $\al$ appears $m$ times. 
The cuspidal modules have the following nice property:

\begin{Lemma} \label{LCuspPower}
\cite[Lemma 3.4]{McN} (cf. \cite[Lemma 6.6]{KRbz}). For any $\al\in R^+$ and $m\in\Z_{> 0}$, we have $L(\al^m)\simeq L(\al)^{\circ m}$. 
\end{Lemma}

If $\la=(\la_1,\dots,\la_r)\in\KP(\al)$ the {\em reduced standard module} is defined to be 
\begin{equation}\label{EBarDe}
\bar\De(\la):=q^{s(\la)} L(\la_1)\circ\dots \circ L(\la_m)
\end{equation}
for a specific degree shift $s(\la)$, whose description will not be important. 
By \cite[Theorem 3.1]{McN} (cf. \cite[7.2, 7.4]{KRbz}), the $H_\al$-module $\bar\De(\la)$ has simple head $L(\la)$, and  in the Grothendieck group of finite dimensional graded $H_\al$-modules, we have 
\begin{equation}\label{EDecN}
[\bar\De(\la)]=[L(\la)]+\sum_{\mu<\la}d_{\la,\mu}[L(\mu)]
\end{equation}
for some coefficients $d_{\la,\mu}\in\Z[q,q^{-1}]$, called the {\em decomposition numbers}. The decomposition numbers depend on the characteristic of the ground field $k$.

 Let $P(\la)$ denote a projective cover of $L(\la)$  in $\mod{H_\al}$. 
For $V\in\mod{H_\al}$ we define the (graded) composition multiplicity $[V:L(\la)]_q:=\DIM\Hom(P(\la),V)\in\Z((q))$. 
The {\em standard module} $\De(\la)$ is defined as the largest quotient of $P(\la)$ all of whose composition factors are of the form $L(\mu)$ with $\mu\leq \la$, see \cite[Corollary 4.13]{Kato}, \cite[Corollary 3.16]{BKM}, \cite[(4.2)]{Kdonkin}. We note that while the irreducible modules $L(\la)$ are all finite dimensional, the standard modules $\De(\la)$ are always infinite dimensional. The standard modules have the usual nice properties:


\begin{Theorem} \label{TStand} {\rm \cite[\S3]{BKM}}
Let $\al\in Q^+$ and $\la,\mu\in\KP(\al)$. 
\begin{enumerate}
\item[{\rm (i)}] $\De(\la)$ has a simple head $L(\la)$, and $[\De(\la):L(\mu)]_q\neq 0$ implies $\mu\leq \la$. 
\item[{\rm (ii)}] We have $\Hom_{H_\al}(\De(\la),\De(\mu))=0$ unless $\la\leq\mu$. 
\item[{\rm (iii)}] For $m\geq 1$, we have $\Ext^m_{H_\al}(\De(\la),\De(\mu))=0$ unless $\la<\mu$. 
\item[{\rm (iv)}] The module $P(\la)$ has a {\em finite} filtration $P(\la)=P_0\supset P_1\supset\dots\supset P_N=0$ with $P_0/P_1\cong \De(\la)$ and $P_r/P_{r+1}\simeq \De(\mu^{(r)})$ for $\mu^{(r)}>\la$, $r=1,2,\dots$. 

\item[{\rm (v)}] Denoting the graded multiplicities of the factors in a $\De$-filtration of $P(\la)$ by $(P(\la):\De(\mu))_q$, we have $(P(\la):\De(\mu))_q=d_{\mu,\la}(q)$. 
\end{enumerate}
\end{Theorem}

To construct the standard modules more explicitly, let us first assume that $\alpha\in R^+$ and explain how to construct the {\em cuspidal standard module} $\De(\al)$. Put $q_\alpha = q^{\alpha \cdot \alpha/2}$.  By \cite[Lemma 3.2]{BKM}, there exists unique (up to isomorphism) indecomposable $H_\alpha$-modules, $\Delta_m(\alpha)$, with $\Delta_0(\alpha)=0$, such that there are short exact sequences
\begin{align*}   
&0 \to q_\alpha^{2(m-1)}L(\alpha) \to \Delta_m(\alpha) \to \Delta_{m-1}(\alpha)\to0, \label{EDeltam} \\
&0 \to q_\alpha^2\Delta_{m-1}(\alpha) \to \Delta_m(\alpha) \to L(\alpha) \to 0. \nonumber
\end{align*}
This yields an inverse system $\cdots \to \Delta_2(\alpha) \to \Delta_1(\alpha) \to \Delta_0(\alpha)$, and we have $\Delta(\alpha) \cong \varprojlim \Delta_m(\alpha)$, see \cite[Corollary 3.16]{BKM}.

Let $m\in \Z_{>0}$. An explicit endomorphism $e_m \in \text{End}_{H_{m\alpha}}(\Delta(\alpha)^{\circ m})^{\text{op}}$ 
is defined in  \cite[Section 3.2]{BKM}, and then 
\begin{equation}\label{EDeAlM}
\Delta(\alpha^m) \cong q_\al^{m(m-1)/2}\Delta(\alpha)^{\circ m}e_m.
\end{equation}
Finally, for an arbitrary $\alpha \in Q^+$ and $\lambda \in \KP(\alpha)$, gather together the equal parts of $\lambda$ to write $\lambda=(\lambda_1^{m_1}, \dots, \lambda_s^{m_s}),$ with $\lambda_1 > \cdots > \lambda_s$.  Then by \cite[(3.5)]{BKM}, 
\begin{equation}\label{EDelta}
\Delta(\lambda) \cong \Delta(\lambda_1^{m_1}) \circ \cdots \circ \Delta(\lambda_s^{m_s}).
\end{equation}

Thus, cuspidal standard modules are building blocks for arbitrary standard modules. We will need some of their additional properties. 
Let $\al\in R^+$. 
If $\lambda \in \KP(\al)$ is minimal such that $\lambda > (\al)$, then by \cite[Lemma 2.6]{BKM}, $\lambda = (\beta, \gamma)$ for positive roots $\beta >\alpha >\gamma$.  In this case, $(\be,\ga)$ is called a \emph{minimal pair} for $\alpha$ and we write $\operatorname{mp}(\al)$ for the set of all such.  The following result proved in \cite[\S\S3.1,4.3]{BKM} describes some of the important properties of $\Delta(\alpha)$. 

\begin{Theorem}\label{TDelta}  Let $\alpha \in R^+$.  Then:\begin{enumerate}
\item  $[\Delta(\alpha):L(\alpha)]_q=1/(1-q_\alpha^2)$ and $[\Delta(\alpha):L(\la)]_q=0$ for $\la\neq (\al)$. 
%
%

\item Let $\catC_\al$ be the category of all modules in $\mod{H_\al}$ all of whose composition factors are $\simeq L(\al)$. 
Any $V \in\catC_\al$ is a finite direct sum of copies of the indecomposable modules $\simeq \Delta_m(\alpha)$ and $\simeq \Delta(\alpha)$. 
Moreover, 
$\Delta(\alpha)$ is a projective cover of $L(\al)$ in $\catC_\al$. In particular, $\Ext^d_{H_\alpha}(\Delta(\alpha),V)=0$ for $d \geq 1$ and  $V\in \catC_\al$.

\item $\End_{H_\alpha}(\Delta(\alpha)) \cong k[x]$ for $x$ in degree $2d_\alpha$.

\item 
There is a short exact sequence 
$0 \to q_\alpha^2 \Delta(\alpha) \to \Delta(\alpha) \to L(\alpha) \to 0.$

\item For $(\beta, \gamma) \in \operatorname{mp}(\alpha)$ there is a short exact sequence
$$0 \to q^{-\beta \cdot \gamma} \Delta(\beta) \circ \Delta(\gamma) \xrightarrow{\phi} \Delta(\gamma) \circ \Delta(\beta) \to [p_{\beta, \gamma} + 1] \Delta(\alpha) \to 0,$$
where $p_{\beta, \gamma}$ is the largest integer $p$ such that $\beta-p\gamma$ is a root.
\end{enumerate}
\end{Theorem}

\begin{Corollary} \label{LGrDimDe} 
Let $\al\in R^+$. The dimensions of the graded components $\De(\al)_d$ are $0$ for  $d\ll0$ and are bounded above by some $N>0$ independent of $d$. 
\end{Corollary}
\begin{proof}
By Theorem~\ref{TDelta}(i), we have 
$\dim_q(\Delta(\alpha)) = \frac{1}{1-q_\alpha^2} \dim_q(L(\alpha)),
$
which implies the result since $L(\al)$ is finite dimensional. 
\end{proof}

\subsection{Endomorphisms of standard modules}\label{SSEnd}

We shall denote by $x_\al$ the degree $2d_\al$ endomorphism of $\De(\al)$ which corresponds to $x$ under the isomorphism $\End_{H_\alpha}(\Delta(\alpha))\cong k[x]$ in Theorem~\ref{TDelta}(iii). This determines $x_\al$ uniquely up to a scalar.  

\begin{Lemma}\label{submodule lemma}
Let $\alpha \in R^+$. Then every non-zero $H_\al$-endomorphism of $\De(\al)$ is injective, and 
every submodule of $\Delta(\alpha)$ is equal to $x_\al^s(\Delta(\alpha))\cong q_\alpha^{2s} \Delta(\alpha)$ for some $s \in \mathbb{Z}_{\geq 0}$.
\end{Lemma}

\begin{proof}
It follows from the construction of $x_\al$ in \cite[Theorem 3.3]{BKM} that $x_\al$ is injective and $x_\al(\Delta(\alpha))\cong q_\alpha^{2} \Delta(\alpha)$. This in particular implies the first statement.

Let $V \subseteq \Delta(\alpha)$ be a submodule and $f: V \to \Delta(\alpha)$ be the natural inclusion.  First, assume that  $V$ is indecomposable.  By Theorem~\ref{TDelta}(ii), up to degree shift, $V$ is isomorphic to $\Delta(\alpha)$ or $\Delta_m(\alpha)$ for some $m \geq 1$.  If $V \simeq \Delta_m(\alpha)$ then 
$\Delta(\alpha)/V$ is infinite dimensional and has a simple head, so by Theorem~\ref{TDelta}(ii) again, $\Delta(\alpha)/V \simeq \Delta(\alpha)$. Then the short exact sequence
$0 \to V \to \Delta(\alpha) \to \Delta(\alpha)/V\to0$ 
splits by projectivity in Theorem~\ref{TDelta}(ii), contradicting indecomposability of $\De(\al)$.  If instead $V \simeq \Delta(\alpha)$, consider the composition
$$\Delta(\alpha) \xrightarrow{\sim} V \xrightarrow{f} \Delta(\alpha).$$
This produces a graded endomorphism of $\Delta(\alpha)$, so that $V = x_\al^s(\Delta(\alpha))$ for some $s \geq 0$.  Since there are inclusions $\Delta(\alpha) \supset x_\al\Delta(\alpha) \supset x_\al^2 \Delta(\alpha) \supset \cdots,$ the general case follows from the case when $V$ is indecomposable.
\end{proof}

Let again $\al\in R^+$. We next consider the standard modules of the form $\De(\al^m)$. We have commuting endomorphisms $X_1,\dots,X_m\in\End_{H_{m\al}}(\De(\al)^{\circ m})$ with $X_r=\id^{\circ (r-1)}\circ x_\al\circ\id^{\circ m-r}$. Moreover, in \cite[Section 3.2]{BKM}, some additional endomorphisms $\partial_1,\dots,\partial_{m-1}\in\End_{H_{m\al}}(\De(\al)^{\circ m})$ are constructed, and it is proved in \cite[Lemmas 3.7-3.9]{BKM} that the algebra $\End_{H_{m\al}}(\De(\al)^{\circ m})^{\op}$ is isomorphic to the nilHecke algebra $NH_m$, with $\partial_1,\dots,\partial_{m-1}$ and (appropriately scaled) $X_1,\dots,X_m$ corresponding to the standard generators of $NH_m$. The element $e_m$ used in (\ref{EDeAlM}) is an explicit idempotent in $NH_m$. We denote by $\La_{\al,m}$ the algebra of symmetric functions  $k[X_1,\dots,X_m]^{S_m}= Z(NH_m)$, with the variables $X_r$ in degree $2d_\al$. 
Note that $\DIM \La_{\al,m}=1/\prod_{r=1}^m(1-q_\alpha^{2r})$. It is known, see e.g. \cite[Theorem 4.4(iii)]{KLM}, that 
\begin{equation}\label{EEMZ}
e_m NH_m e_m=e_m\La_{\al,m}\cong \La_{\al,m}.
\end{equation}

\begin{Theorem}\label{TDeSemiCusp}  Let $\alpha \in R^+$ and $m\in \Z_{>0}$.  Then:

\begin{enumerate}
\item For any $\la\in\KP(m\al)$, we have 
$[\Delta(\alpha^m):L(\la)]_q=\de_{\la,(\al^m)}/\prod_{r=1}^m(1-q_\alpha^{2r})$. 
%
%

\item The module $\Delta(\alpha^m)$ is a projective cover of $L(\al^m)$ in the category 
of all modules in $\mod{H_\al}$ all of whose composition factors are $\simeq L(\al^m)$. 

\item $\End_{H_\alpha}(\Delta(\alpha)) \cong \La_{\al,m}$. 

\item 
Every submodule of $\De(\al^m)$ is isomorphic to $q^d\De(\al^m)$ for some $d\in\Z_{\geq 0}$, and every non-zero $H_{m\al}$-endomorphism of $\De(\al^m)$ is injective.
\end{enumerate}
\end{Theorem}
\begin{proof}
Part (i) is \cite[Lemma 3.10]{BKM}, and part (ii) follows  from \cite[Lemma 4.11]{Kdonkin}, since $(\al^m)$ is minimal in $\KP(\al)$ by convexity. By (i) and (ii), we have that $\DIM\End_{H_{m\al}}(\De(\al^{m}))=1/\prod_{r=1}^m(1-q_\alpha^{2r})$.

(iii) We have that $NH_m=\End_{H_{m\al}}(\De(\al)^{\circ m})^\op$ acts naturally on $\De(\al)^{\circ m}$ on the right, and so $\La_{\al,m}=Z(NH_m)$ acts naturally on $\De(\al^m)=\De(\al)^{\circ m}e_m$. This defines an embedding $\La_{\al,m} \to \End_{H_{m\al}}(\De(\al^{m}))$. This embedding must be an isomorphism by dimensions. 

(iv) In view of Lemma~\ref{submodule lemma}, every non-zero $f\in k[X_1,\dots,X_m]\subseteq NH_m=\End_{H_{m\al}}(\De(\al)^{\circ m})^\op$ acts as an injective linear operator on $\De(\al)^{\circ m}$. 
The result now follows from (\ref{EEMZ}) and (ii).  
\end{proof}

Finally, we consider a general case. Let $\al\in Q^+$ and $\lambda=(\lambda_1^{m_1}, \dots, \lambda_s^{m_s}) \in \KP(\alpha)$ with $\lambda_1 > \cdots > \lambda_s$.  We have a natural embedding 
\begin{equation}\label{ELaGen}
\La_{\la_1,m_1}\otimes\dots\otimes \La_{\la_s,m_s}\to \End_{H_\al}(\De(\la)),\ f_1\otimes\dots\otimes f_s\mapsto f_1\circ\dots\circ f_s.
\end{equation}

\begin{Theorem} \label{TInjGeneral} 
Let $\al\in Q^+$ and $\lambda=(\lambda_1^{m_1}, \dots, \lambda_s^{m_s}) \in \KP(\alpha)$ with $\lambda_1 > \cdots > \lambda_s$. Then $$\End_{H_\al}(\De(\la))\cong \La_{\la_1,m_1}\otimes\dots\otimes \La_{\la_s,m_s}$$ via  (\ref{ELaGen}), and every non-zero $H_\al$-endomorphism of $\De(\la)$ is injective.  
\end{Theorem}
\begin{proof}
It is easy to see from Theorem~\ref{TDeSemiCusp}(iv) that every non-zero endomorphism in the image of the embedding (\ref{ELaGen}) is injective. To see that there are no other endomorphisms, we first use adjointness of $\End$ and $\Res$ to see that $\End_{H_\al}(\De(\la))$ is isomorphic to  
$$
\Hom_{H_{m_1\la_1}\otimes\dots\otimes H_{m_s\la_s}}(\De(\la_1^{m_1})\boxtimes\dots\boxtimes \De(\la_s^{m_s}), \Res^\al_{m_1\la_1,\dots,m_s\la_s}\De(\la)),
$$
and then note that by the Mackey Theorem, as in \cite[Lemma 3.3]{McN}, we we have $\Res^\al_{m_1\la_1,\dots,m_s\la_s}\De(\la)\cong \De(\la_1^{m_1})\boxtimes\dots\boxtimes \De(\la_s^{m_s})$. 
\end{proof}


\section{Proof of Theorem A}\label{SThA}

We give the proof of Theorem A based on the recent work of Kashiwara-Park \cite{KP}. Our original proof was different and relied on some unpleasant  computation for non-simply-laced types. For simply laced types however, our original proof is very simple and elementary, and so we give it later in this section, too.

\subsection{Proof of Theorem A modulo a hypothesis}\label{SModulo}
The following hypothesis concerns a key property of cuspidal standard modules and is probably true beyond finite Lie types:

\begin{Hypothesis} \label{Hypothesis}
Let $\al$ be a positive root of height $n$ and $1\leq r\leq n$. Then upon restriction to the subalgebra $k[x_r]\subseteq H_\al$, the module $\Delta(\alpha)$ is free of finite rank.
\end{Hypothesis}

The goal of this subsection is to prove  Theorem A assuming the hypothesis. In \S\ref{SSKP} the hypothesis will be proved using results of  Kashiwara and Park, while in \S\ref{SSL}  we will give a more elementary proof for simply laced types.

\begin{Lemma} \label{free y module}
Hypothesis~\ref{Hypothesis} is equivalent to the property that $x_1,\dots,x_n$ act by injective linear operators on $\De(\al)$.
\end{Lemma}

\begin{proof}
The forward direction is clear. For the converse, assume that $x_r$ acts injectively on $\De(\al)$. We construct a finite basis for $1_\bi \Delta(\al)$ as a $k[x_r]$-module for every $\bi \in \words^\al$. Let $m:=\deg(x_r 1_\bi)$. For every $a=0,1,\dots,m-1$, let $d_a$ be minimal with $d_a\equiv a\pmod{m}$ and $1_\bi \De(\al)_{d_a}\neq 0$. Pick a linear basis of $\oplus_{a=0}^{m-1}1_\bi \De(\al)_{d_a}$ and note that the $k[x_r]$-module generated by the elements of this basis is free. Factor out this $k[x_r]$-submodule, and repeat. The process will stop after finitely many steps, thanks to Corollary~\ref{LGrDimDe}. 
\end{proof}

While Hypothesis~\ref{Hypothesis} claims that every $k[x_r]$ acts freely on $\De(\al)$, no $k[x_r,x_s]$ does:

\begin{Lemma}\label{f kills} Let $\alpha \in R^+$ be a root of height $n>1$.  Then, for every vector $v \in \Delta(\alpha)$, and distinct $r,s \in \{1, \cdots , n\}$, there is a polynomial $f \in k[x,y]$ such that $f(x_r,x_{s})  v = 0$. 
\end{Lemma}
\begin{proof}
We may assume $v$ is a homogenous weight vector. By Corollary~\ref{LGrDimDe}, the dimensions of the graded components of $\Delta(\alpha)$ are uniformly bounded.  The result follows, as the number of linearly independent degree $d$ monomials in $x,y$ grows without bound. 
\end{proof}

One can say more about the polynomial $f$ in the lemma, see for example Proposition~\ref{Px-y}.

Now, let $\al\in Q^+$ be arbitrary 
of height $n$, and $\lambda = (\lambda_1 \geq \cdots \geq \lambda_l) \in \KP(\al)$.  Setting $S_\lambda := S_{\height(\lambda_1)} \times \dots \times S_{\height(\lambda_l)} \subset S_n,$ integers $r, s \in\{1,\dots, n\}$ are called  \emph{$\lambda$-equivalent}, written $r\sim_\la s$, if they belong to the same orbit of the action of $S_\lambda$ on $\{1, \dots, n\}$. Finally, recalling the idempotents (\ref{EId}), we  write $1_\lambda:=1_{\la_1,\dots,\la_l}$.

\begin{Lemma} \label{equivalence lemma}
Let $\al\in Q^+$, $n=\height(\al)$, and $\la\not\geq \mu$ be elements of $\KP(\al)$.  If $w \in S_n$ satisfies $1_\la \tau_w 1_\mu \neq 0$ then there exists some $1 \leq r <n$ such that $r\sim_\la r+1$, but $w^{-1}(r)\not\sim_\mu w^{-1}(r+1)$.
\end{Lemma}

\begin{proof}
Write $\la = (\la_1 \geq \cdots \geq \la_l)$ and $\mu=(\mu_1 \geq \cdots \geq \mu_m)$. The assumption $1_\la \tau_w 1_\mu \neq 0$ implies that $\bi^\la=w\cdot \bi^\mu$ for some $\bi^\la\in I^{\la_1,\dots,\la_l}$ and $\bi^\mu\in I^{\mu_1,\dots,\mu_m}$. 
Write $\bi^\la := \bi_1^\la \cdots \bi_l^\la$ with $\bi_a^\la \in I^{\la_a}$ for all $a$, and $\bi^\mu := \bi_1^\mu \cdots \bi_m^\mu$ with $\bi_b^\mu \in I^{\mu_b}$ for all $b$. Assume for a contradiction that for every $1 \leq r < n$ we have  $r\sim_\la r+1$  implies that $w^{-1}(r)\sim_\mu w^{-1}(r+1)$.  Then there is a partition $\{1,\dots,l\}=\bigsqcup_{b=1}^m A_b$ such that $\mu_b =\sum_{a\in A_b} \lambda_{a}$ for all $b=1,\dots, m$. By convexity, cf. \cite[Lemma 2.4]{BKM}, we have $\min\{\lambda_{a}\mid a\in A_b\}\leq \mu_b \leq 
\max\{\lambda_{a}\mid a\in A_b\}$.  This implies $\la \geq \mu$.
\end{proof}

\begin{Theorem} \label{main theorem}
Let $\alpha \in Q^+$ and $\lambda,\mu\in \KP(\alpha)$.  If $\lambda \neq \mu$, then 
\begin{equation*}
\textup{Hom}_{H_\alpha}(\Delta(\lambda), \Delta(\mu)) = 0.
\end{equation*}
\end{Theorem}

\begin{proof}
Let $n=\height(\al)$ and write $\lambda = (\lambda_1 \geq \cdots \geq \lambda_l)$ and $\mu = (\mu_1 \geq \cdots \geq \mu_m)$. 
It suffices to prove that 
$$\textup{Hom}_{H_\alpha}(\Delta(\lambda_1) \circ \cdots \circ \Delta(\lambda_l), \Delta(\mu_1) \circ \cdots \circ \Delta(\mu_m)) = 0.$$  
If not, let $\phi$ be a nonzero homomorphism.  By Theorem~\ref{TStand}(ii), we may assume that $\lambda < \mu$.  
Using Lemma~\ref{f kills}, pick a generator $v \in \Delta(\lambda_1) \circ \cdots \circ \Delta(\lambda_l)$ such that $v=1_\la v$ and for any $r\sim_\la r+1$, there is a non-zero polynomail $f\in k[x,y]$ with $f(x_r, x_{r+1})v=0$.  Then $f(x_r, x_{r+1})\phi(v) = 0$  as well.

Denote by $S^\mu$ the set of shortest length coset representatives for $S_n/S_\mu.$   Then, we can write $\phi(v) = \sum_{w \in S^\mu} \tau_{w} \otimes v_{w}$ for some $v_w \in \Delta(\mu_1) \otimes \cdots \otimes \Delta(\mu_{m})$.  
Since $\phi(v) = 1_\la \phi(v)$ and $1_\mu v_w=v_w$, we have that $1_\la \tau_w 1_\mu \neq 0$ whenever $v_w\neq 0$. 
In particular, if  $u \in S^\mu$ is an element of maximal length such that $v_u \neq 0$, then by Lemma~\ref{equivalence lemma}, 
$r\sim_\la r+1$ and  $u^{-1}(r)\not\sim_\mu u^{-1}(r+1)$ for some $1\leq r<n$.  

Now, we have:
\begin{align*}
f(x_r,x_{r+1}) \phi(v) &= f(x_r,x_{r+1}) \sum_{w \in S^\mu} \tau_w \otimes v_w \\
&= f(x_r,x_{r+1}) \tau_u \otimes v_u + \sum_{w \neq u} f(x_r,x_{r+1}) \tau_w \otimes v_ w \\
&=\tau_u \otimes f(x_{u^{-1}(r)},x_{u^{-1}(r+1)}) v_u + 
(*),
\end{align*}
where $(*)$ is a sum of elements of the form $\tau_w\otimes v_w'$ with $v_w'\in \Delta(\mu_1) \otimes \cdots \otimes \Delta(\mu_{m})$ and $w\in S^\mu\setminus\{u\}$. 
 The last equality follows because in $H_\al$ for all $1 \leq t \leq n$ and $w\in S_n$, we have that 
$x_t \tau_w = \tau_w x_{w^{-1}(t)} + (**),$
where $(**)$ is a linear combination of elements of the form $\tau_y $ with  $y \in S_n$ being Bruhat smaller than $w$.

Since $u^{-1}(r)\not\sim_\mu u^{-1}(r+1)$,  
there are distinct integers $a,b \in \{1, \dots, m\}$ and integers $1\leq c\leq \height(\mu_a)$ and $1\leq d\leq \height(\mu_b)$ such that for any pure tensor $v= v^1 \otimes \cdots \otimes v^m \in \Delta(\mu_1) \otimes \cdots \otimes \Delta(\mu_m)$, and $s, t \in \mathbb{Z}_{\geq 0}$, we have 
$$x_{u^{-1}(r)}^{s}x_{u^{-1}(r+1)}^{t} v = v^1 \otimes \cdots \otimes x_{c}^{s}v^{a} \otimes \cdots \otimes x_{d}^{t} v^{b} \otimes \cdots \otimes v^m.$$
By Hypothesis~\ref{Hypothesis}, $f(x_{u^{-1}(r)},x_{u^{-1}(r+1)}) v_u\neq 0$. Hence $f(x_r,x_{r+1}) \phi(v) \neq 0$ giving a contradiction. 
\end{proof}

\subsection{Proof of the Hypothesis using Kashiwara-Park Lemma}\label{SSKP}
We begin with a key lemma which follows immediately from the results of \cite{KP}:

\begin{Lemma} \label{LKP} 
Let $\al\in R^+$, $n=\height(\al)$ and $i\in I$. Define 
$${\mathfrak p}_{i,\al}:=\sum_{\bi\in I^\al}\bigg(\prod_{r\in[1,n], i_r=i}x_r\bigg)1_\bi.
$$ 
Then ${\mathfrak p}_{i,\al}\De(\al)\neq 0$. 
\end{Lemma}
\begin{proof}
This follows from \cite[Definition 2.2(b)]{KP} and \cite[Proposition 3.5]{KP}.
\end{proof}

\begin{Theorem}
Let $\alpha \in R^+$ have height $n$.  Then, $x_r^m   v \neq 0$ for all $1 \leq r \leq n$,  $m \in \mathbb{Z}_{\geq 0}$, and nonzero  $v \in \Delta(\alpha)$.  In particular, Hypothesis~\ref{Hypothesis} holds.
\end{Theorem}

\begin{proof} The `in particular' statement follows from Lemma~\ref{free y module}. 

We may assume that $v$ is a weight vector of some weight $\bi$. Let $i=i_r$. The element ${\mathfrak p}_{i,\al}$ defined in Lemma~\ref{LKP}  is central by Theorem~\ref{TCenter}. By Lemma~\ref{LKP} and Theorem~\ref{TInjGeneral}, the multiplication with ${\mathfrak p}_{i,\al}$ on $\De(\al)$ is injective, so multiplication with ${\mathfrak p}_{i,\al}^m$ is also injective. But ${\mathfrak p}_{i,\al}^m$ involves $x_r1_\bi$, so
 $0\neq {\mathfrak p}_{i,\al}^m v=hx_r^mv$ for some $h\in H_\al$, and the theorem follows. 
\end{proof}

\subsection{Elementary proof of the Hypothesis for simply laced types}
\label{SSL}
Throughout this subsection, we assume that the root system $R$ is of  (finite) $ADE$ type.  Let $\alpha = a_1\alpha_1+\cdots +a_l \alpha_l \in Q^+$ and $n=\height(\al)$.  Pick a permutation $(i_1, \dots, i_l)$ of $(1, \dots, l)$ with $a_{i_1}>0$, and define $\bi := i_1^{a_{i_1}} \cdots i_l^{a_{i_l}} \in I^\alpha$.  Then the stabalizer of $\bi$ in $S_n$ is the standard parabolic subgroup $S_{\bi} := S_{a_{i_1}} \times \cdots \times S _{a_{i_l}}$.  Let $S^{\bi}$ be a set of coset representatives for $S_n/S_{\bi}$.  Then by Theorem~\ref{TCenter}, the element
\begin{equation} \label{central element}
z=z_{\bi} := \sum_{w \in S^{\bi}} (x_{w(1)} + \cdots +x_{w(a_{i_1})})1_{w \cdot \bi}
\end{equation}
is central of degree $2$ in $H_\alpha$.  For any $1 \leq r \leq n$, note that 
\begin{equation} \label{solve yz}
a_{i_1} x_r = z - \sum_{w \in S^{\bi}} ((x_{w(1)}-x_r)+\cdots+(x_{w(a_{i_1})}-x_r))1_{w \cdot \bi}.
\end{equation}

Let $H_{\alpha}'$ be the subalgebra of $H_\alpha$ generated by
\begin{equation*}
\{1_{\bi} \mid  \bi \in I^\alpha \} \cup \{\tau_r \mid  1 \leq r < n \} \cup \{ x_r-x_{r+1} \mid  1 \leq r < n \}.
\end{equation*}

For the reader's convenience, we reprove a lemma from \cite[Lemma 3.1]{BKOld}:

\begin{Lemma} \label{LH'}
Let $\al$, $\bi$, and $z$ be as above. Then:
\begin{enumerate}
\item[{\rm (i)}] $\{(x_1-x_2)^{m_1} \cdots (x_{n-1}-x_{n})^{m_{n-1}} \tau_w 1_{\bi} \mid  m_r \in \mathbb{Z}_{\geq 0}, w \in S_n, \bi\in I^\alpha\}
$ 
is a basis for $H_{\alpha}'$.
\item[{\rm (ii)}] If $a_{i_1}\cdot 1_k\neq 0$ in $k$, then 
there is an algebra isomorphism
\begin{equation} \label{tensor algebra iso}
H_\alpha \cong H_\alpha' \otimes k[z].
\end{equation}
\end{enumerate}
\end{Lemma}
\begin{proof}
In view of the basis (\ref{EBasis}), part (i) follows on checking that
the span of the given monomials is closed under multiplication, which follows from the defining relations.
For (ii), note using (\ref{solve yz}) that the natural multiplication map
$k[z] \otimes H_\alpha' \rightarrow H_\alpha$ is surjective.
It remains to observe 
that 
the two algebras have the same graded dimension.
\end{proof}

Let $\alpha$ now be a positive root. Then one can always find an  index $i_1$ with $a_{i_1} \cdot 1_k \neq 0$, so in this case we always have \eqref{tensor algebra iso} for an appropriate choice of $\bi$. We always assume that this choice has been made. 
Following \cite{BKOld}, we can now present another useful description of the cuspidal standard module $\De(\al)$. Denote by $L'(\alpha)$ the restriction of the cuspidal irreducible module $L(\alpha)$ from $H_\alpha$ to $H_\alpha'$.

\begin{Lemma} \label{LNewDelta}
Let $\al\in R^+$. 
\begin{enumerate}
\item[{\rm (i)}] $L'(\al)$ is an irreducible $H_\al'$-module. 
\item[{\rm (ii)}] $\Delta(\alpha) \cong H_\alpha \otimes_{H_{\alpha}'} L'({\alpha})$. 
\item[{\rm (iii)}] The element $z$ acts on $\Delta(\alpha)$ freely.
\end{enumerate}
\end{Lemma}
\begin{proof}
Note that $z$ acts as zero on $L(\al)$, which implies (i) in view of (\ref{tensor algebra iso}). Moreover, it is now easy to see that $H_\alpha \otimes_{H_{\alpha}'} L'({\alpha})$ has a filtration with the subfactors isomorphic to $q^{2d}L(\al)$ for $d=0,1,\dots$. 
Furthermore, by Frobenius Reciprocity and (i), the module $H_\alpha \otimes_{H_{\alpha}'} L'({\alpha})$ has simple head $L(\al)$. Now (ii) follows from Theorem~\ref{TDelta}(ii). Finally, (iii) follows from (ii) and (\ref{tensor algebra iso}).
\end{proof}

Using the description of $\De(\al)$ from Lemma~\ref{LNewDelta}(ii), we can now establish  Hypothesis~\ref{Hypothesis}: 

\begin{Theorem} \label{y acts freely}
Let  $\al\in R^+$ and $\{v_1, \dots, v_N\}$ be a $k$-basis of $L'(\alpha)$. Then the $k[x_r]$-module $\Delta(\alpha)$ is free with basis $\{1 \otimes v_1, \dots, 1 \otimes v_N\}$. 
In particular, Hypothesis~\ref{Hypothesis} holds for simply laced types. 
\end{Theorem}

\begin{proof}
 By \eqref{solve yz}, we can write
$ x_r = \frac{1}{a_i}z+(*),
$
 where $(*)$ is an element of $H_\alpha'$.  For each $1\leq m \leq N$, we have 
$$x_r^b ( 1 \otimes v_m) = \left(\frac{1}{a_k}\right)^b z^b \otimes v_m + (**),$$
where $(**)$ is a linear combination of terms of the form $z^c \otimes v_t$ with $c<b$. So $\{1 \otimes v_1, \dots, 1 \otimes v_N\}$ is a basis of the free $k[x_r]$-module $\Delta(\alpha)$. \end{proof}

The following strengthening of Lemma~\ref{f kills} is not needed for the proof of  Theorem A, but we include it for completeness.

\begin{Proposition} \label{Px-y}
Let $\alpha \in R^+$ and $n=\height(\al)$.  For any $1\leq r,s\leq n$, there is $d\in\Z_{>0}$ such that $(x_r-x_s)^d$ annihilates $\De(\al)$. 
\end{Proposition}
\begin{proof}
Pick $d$ such that $(x_r-x_s)^d$ annihilates $L(\al)$. Since $\De(\al)=H_\al\otimes_{H_\al'} L'(\al)$ is spanned by vectors of the form $z^m\otimes v'$ with $m\in\Z_{\geq 0}$ and $v'\in L'(\al)$, it suffices to note that $(x_r-x_s)^d(z^m\otimes v')=z^m\otimes (x_r-x_s)^dv'=0$.
\end{proof}


\section{Reduction modulo $p$}\label{SRed}\label{SRedModP}
\subsection{Changing scalars}
In this subsection we develop a usual formalism of modular representation theory for KLR algebras. There will be nothing surprising here, but we need to exercise care since we work with infinite dimensional algebras and often with infinite dimensional modules. 

From now on, we will work with different ground rings, so our notation needs to become more elaborate. Recall that the KLR algebra $H_\al$ is defined over an arbitrary commutative unital  ring $k$, and to emphasize which $k$ we are working with, we will use the notation $H_{\al,k}$ or $H_{\al;k}$. 
In all our notation we will now use the corresponding index.
For example, if $k$ is a field, we now denote the irreducible cuspidal modules over $H_{\al,k}$ by $L(\al)_k$.

Let $p$ be a fixed prime number, and $F:=\Z/p\Z$ be the prime field of characteristic~$p$. We will use the $p$-modular system $(F,R,K)$ with $R=\Z_p$ and $K=\Q_p$. Oftentimes (when we can avoid  lifting idempotents) we could get away with $R=\Z$, $K$ any field of characteristic zero, and $F$ any field of characteristic~$p$. 

Recall from Section~\ref{SPrel} that for a left Noetherian graded algebra $H$, we denote by $\mod{H}$ the category of finitely generated graded $H$-modules, for which we have the groups $\ext^i_H(V,W)$ and $\Ext^i_H(V,W)$. To deal with change of scalars in Ext groups, we will use the following version of the Universal Coefficient Theorem:

\begin{Theorem} 
{\bf (Universal Coefficient Theorem)}
Let $V_R,W_R$ be modules in $\mod{H_{\al,R}}$, free as $R$-modules, and $k$ be an $R$-algebra. 
Then for all $j\in \Z_{\geq 0}$ there is an exact sequence of (graded) $k$-modules
\begin{align*}
0
&\rightarrow \Ext^j_{H_{\al,R}}(V_R,W_R)\otimes_R k\rightarrow \Ext^j_{H_{\al,k}}(V_R\otimes_R k,W_R\otimes_R k)
\\
&\rightarrow \Tor^R_1\big(\Ext^{j+1}_{H_{\al,R}}(V_R,W_R)\, ,\,k\big)\rightarrow 0.
\end{align*}
In particular, 
$$\Ext^j_{H_{\al,R}}(V_R,W_R)\otimes_R K\cong  \Ext^j_{H_{\al,K}}(V_R\otimes_R K,W_R\otimes_R K).$$
\end{Theorem}
\begin{proof}
This is known. Apply $\Hom_{H_{\al,R}}(-,W_R)$ to a free resolution of $V_R$ to get a complex $C_\bullet$ of free (graded) $R$-modules with finitely many generators in every graded degree. Now follow the proof of \cite[Theorem 8.22]{Rotman}. The second statement follows from the first since $K$ is a flat $R$-module. 
\end{proof}

We need another standard result, whose proof is omitted. 

\begin{Lemma} \label{LExtExt} 
Let $k=K$ or $F$, $V_R,W_R\in\mod{H_{\al,R}}$ be free as $R$-modules, and 
$$0\to W_R\stackrel{\iota}{\longrightarrow} E_R\stackrel{\pi}{\longrightarrow} V_R\to 0$$
 be the extension corresponding to a class $\xi\in\Ext^1_{H_{\al,R}}(V_R,W_R)$. Identifying  $\Ext^1_{H_{\al,R}}(V_R,W_R)\otimes_R k$ with a subgroup of $\Ext^1_{H_{\al,k}}(V_R\otimes_R k,W_R\otimes_R k)$,  we have that 
$$
 0\to W_R\otimes_Rk\stackrel{\iota\otimes\id_k}{\longrightarrow} E_R\otimes_Rk\stackrel{\pi\otimes\id_k}{\longrightarrow} V_R\otimes_Rk\to 0
 $$ is the extension corresponding to a class $\xi\otimes 1_k\in\Ext^1_{H_{\al,R}}(V_R,W_R)\otimes_Rk$. 
\end{Lemma}

Let $k=K$ or $F$, and $V_k$ be an $H_{\al,k}$-module. We say that an $H_{\al,R}$-module $V_R$ is an {\em $R$-form of $V_k$} if every graded component of $V_R$ is free of finite rank as an $R$-module and, identifying $H_{\al,R}\otimes_R k$ with $H_{\al,k}$, we have  $V_R\otimes_R k\cong V_k$ as $H_{\al,k}$-modules. 
If $k=K$, by a {\em full lattice} in $V_K$ we mean an $R$-submodule $V_R$ of $V_K$ such that every graded component $V_{d,R}$ of $V_R$ is a finite rank free $R$-module  which generates the graded component $V_{d,K}$ as a $K$-module. If $V_R$ is an $H_{\al,R}$-invariant full lattice in $V_K$, it is an $R$-form of $V_K$. Now we can see that every $V_K\in \mod{H_{\al,K}}$ has an $R$-form: pick $H_{\al,K}$-generators $v_1,\dots,v_r$ and define $V_R:=H_{\al,R}\cdot v_1+\dots+H_{\al,R}\cdot v_1$. 

The projective indecomposable modules over $H_{\al,F}$ have projective $R$-forms. Indeed, any $P(\la)_F$ is of the form $H_{\al,F}e_{\la,F}$ for some {\em degree zero} idempotent $e_{\la,F}$. By the Basis Theorem, the degree zero component $H_{\al,F,0}$ of $H_{\al,F}$ is defined over $R$; more precisely,   we have 
$H_{\al,k,0}=H_{\al,R,0}\otimes_R k$ for $k=K$ or $F$. Since $H_{\al,F,0}$ is finite dimensional, by the classical theorem on lifting idempotents \cite[(6.7)]{CR}, there exists an idempotent $e_{\la,R}\in H_{\al,R,0}$ such that $e_{\la,F}=e_{\la,R}\otimes 1_F$, and $P(\la)_R:=H_{\al,R}e_{\la,R}$ is an $R$-form of $P(\la)_F$. The notation $P(\la)_R$ will be reserved only for this specific  $R$-form of $P(\la)_F$.  
Note that, while the $H_{\al,R}$-module $P(\la)_R$ is indecomposable, it is not in general true that  $P(\la)_R\otimes_R K\cong P(\la)_K$, see Lemma~\ref{LAdj} for more information. 

Let $V_K\in\mod{H_{\al,K}}$ and $V_R$ be an $R$-form of $V_K$. 
The $H_{\al,F}$-module $V_R\otimes_R F$ is called a {\em reduction modulo $p$} of $V_K$. Reduction modulo $p$ in general depends on the choice of $V_R$. However, as usual, we have:

\begin{Lemma} \label{LRedIndep} 
If $V_{K}\in\mod{H_{\al,K}}$ and $V_R$ is an $R$-form of $V_K$, then for any $\la\in\KP(\al)$, we have 
$$
[V_R\otimes_R F:L(\la)_F]_q=\DIM \Hom_{H_{\al,K}}(P(\la)_R\otimes_R K,V_K).
$$
In particular, the composition multiplicities $[V_R\otimes_R F:L(\la)_F]_q$ 
are independent of the choice of an $R$-form $V_R$. 
\end{Lemma}
\begin{proof}
We have $[V_R\otimes_R F:L(\la)_F]_q=\DIM\Hom_{H_{\al,F}}(P(\la)_F,V_R\otimes_R F)$. 
By the Universal Coefficient Theorem, 
$$
\Hom_{H_{\al,F}}(P(\la)_F,V_R\otimes_R F)\cong \Hom_{H_{\al,R}}(P(\la)_R,V_R)\otimes_R F.
$$
Moreover, note that $\Hom_{H_{\al,R}}(P(\la)_R,V_R)$ is $R$-free of (graded) rank equal to $\DIM \Hom_{H_{\al,R}}(P(\la)_R,V_R)\otimes_R k$ for $k=F$ or $K$. Now, by the Universal Coefficient Theorem again, we have that 
$$
\DIM\Hom_{H_{\al,R}}(P(\la)_R,V_R)\otimes_R K
=\DIM
\Hom_{H_{\al,K}}(P(\la)_R\otimes_R K,V_R\otimes_R K),
$$
which completes the proof, since $V_R\otimes_R K\cong V_K$. 
\end{proof}

Our main interest is in reduction modulo $p$ of the irreducible $H_{\al,K}$-modules $L(\la)_K$. 
Pick a non-zero homogeneous vector $v\in L(\la)_K$ and define $L(\la)_R:=H_{\al,R}\cdot v$. Then $L(\la)_R$ is an $H_{\al,R}$-invariant full lattice in $L(\la)_K$, and reducing modulo $p$, we get an $H_{\al,F}$-module $L(\la)_R\otimes_R F$. In general, $L(\la)_R\otimes_R F$ is not $L(\la)_F$, although this happens  `often', for example for cuspidal modules:

\begin{Lemma} \label{LRedCusp} 
\cite[Proposition 3.20]{KSing} 
Let $\al\in R^+$. Then  $L(\al)_R\otimes_R F\cong L(\al)_F$. 
\end{Lemma}

To generalize this lemma to irreducible modules of the form $L(\al^m)$, we need to observe that induction and restriction commute with extension of scalars. More precisely, 
for $\be_1,\dots,\be_m\in Q^+$, $\al=\be_1+\dots+\be_m$, and any ground ring $k$, we denote by $H_{\be_1,\dots,\be_m;k}$ the algebra $H_{\be_1,k}\otimes_k \dots\otimes_k H_{\be_m,k}$ identified as usual with a (non-unital) subalgebra $1_{\be_1,\dots,\be_m;k}H_{\al,k}1_{\be_1,\dots,\be_m;k}\subseteq H_{\al,k}$.

\begin{Lemma} \label{LIndScal} 
Let $V_R\in \mod{H_{\be_1,\dots,\be_m;R}}$ and $W_R\in\mod{H_{\al,R}}$. Then for any $R$-algebra $k$, there are natural isomorphisms of $H_{\al,k}$-modules 
$$(\Ind_{\be_1,\dots,\be_m}^\al V_R)\otimes_R k\cong \Ind_{\be_1,\dots,\be_m}^\al (V_R\otimes_R k)$$ 
and of $H_{\be_1,\dots,\be_m;k}$-modules 
$$(\Res_{\be_1,\dots,\be_m}^\al W_R)\otimes_R k\cong \Res_{\be_1,\dots,\be_m}^\al (W_R\otimes_R k).$$ 
\end{Lemma}

Let $\al\in R^+$ and $m\in \Z_{>0}$. If $k$ is a field, 
by Lemma~\ref{LCuspPower}, we have $L(\al^m)_k\simeq  L(\al)_k^{\circ m}$. By Lemma~\ref{LIndScal},  $L(\al^m)_R:=(L(\al)_R)^{\circ m}$ satisfies $L(\al^m)_R\otimes_R k\simeq L(\al^m)_k$ for $k=K$ or $F$. Taking into account Lemma~\ref{LRedIndep}, we get:

\begin{Lemma} 
Let $\al\in R^+$ and $m\in \Z_{>0}$. Then reduction modulo $p$ of $L(\al^m)_K$ is $L(\al^m)_F$.
\end{Lemma}

It was conjectured in \cite[Conjecture 7.3]{KRbz} that reduction modulo $p$ of $L(\la)_K$ is always $L(\la)_F$, but counterexamples are given in \cite{Williamson} (see also \cite[Example 2.16]{BKM}). Still, it is important to understand when we have $L(\la)_R\otimes_R F\cong L(\la)_F$:

\begin{Problem}\label{Problem}
Let $\al\in Q^+$. 
\begin{enumerate}
\item[{\rm (i)}] If $\la\in\KP(\al)$, determine when $L(\la)_R\otimes_R F\cong L(\la)_F$. 
\item[{\rm (ii)}] We say that {\em James' Conjecture has positive solution} (for $\al$) if the isomorphism in (i) holds for all $\la\in\KP(\al)$. 
Determine the minimal lower bound $p_\al$ on $p=\cha F$ so that  James' Conjecture has positive solution 
for all
$p\geq p_\al$. 
\end{enumerate}
\end{Problem}

At least, we always have:

\begin{Lemma} \label{LAdj} 
Let $\al\in Q^+$ and $\la\in \KP(\al)$. Then in the Grothendieck group of finite dimensional $H_{\al,F}$-modules we have
\begin{equation}\label{ERed}
[L(\la)_R\otimes_R F]= [L(\la)_F]+\sum_{\mu<\la}a_{\la,\mu}[L(\mu)_F]
\end{equation}
for some bar-invariant Laurent polynomials $a_{\la,\mu}\in\Z[q,q^{-1}]$. Moreover, 
$$
P(\la)_R\otimes_R K\cong P(\la)_K \oplus \bigoplus_{\mu>\la}a_{\mu,\la}P(\mu)_K.
$$ 
\end{Lemma}
\begin{proof}
Let $k=K$ or $F$ and consider the  reduced standard module $\bar\De(\la)_k$, see (\ref{EBarDe}). In view of  (\ref{EDecN}), we can write  
$$[L(\la)_k]:=[\bar\De(\la)_k]+\sum_{\mu<\la}f_{\la,\mu}^k[\bar\De(\mu)_k]$$
for some $f_{\la,\mu}^k\in\Z[q,q^{-1}]$. Using Lemmas~\ref{LIndScal},~\ref{LRedCusp} and induction on the bilexicographical order on $\KP(\la)$, we now deduce that the equation (\ref{ERed}) holds for some, not necessarily bar-invariant, coefficients $a_{\la,\mu}\in\Z[q,q^{-1}]$. Then we also have 
$$
\CH (L(\la)_R\otimes_R F)= \CH(L(\la)_F)+\sum_{\mu<\la}a_{\la,\mu}\CH (L(\mu)_F).
$$
Since reduction modulo $p$ preserves formal characters, the left hand side is bar-invariant. Moreover, every $\CH (L(\mu)_F)$ is bar-invariant. This implies that the coefficients $a_{\la,\mu}$ are also bar-invariant, since by \cite[Theorem 3.17]{KL1}, the formal characters $\{\CH L(\nu)_F\mid \nu\in \KP(\al)\}$ are linearly independent. 

Finally, for any $\mu\in\KP(\la)$, we have 
$$a_{\mu,\la}=
\DIM \Hom_{H_{\al,K}}(P(\la)_R\otimes_R K,L(\mu)_K),
$$ 
thanks to by Lemma~\ref{LRedIndep}. This 
 implies the second statement.  
\end{proof}

\begin{Remark} 
{\rm 
For $k=K$ and $F$, denote by $d_{\la,\mu}^k$, the corresponding decomposition numbers, see (\ref{EDecN}), and consider the {\em decomposition matrices} $D^k:=(d_{\la,\mu}^k)_{\la,\mu\in\KP(\al)}$. Setting $A:=(a_{\la,\mu})_{\la,\mu\in\KP(\al)}$, we have  $D^F=D^KA$. So the matrix $A$ plays the role of the {\em adjustment matrix} in the classical James' Conjecture \cite{James}. Note that James' Conjecture has positive solution in the sense of Problem~\ref{Problem} if and only if $A$ is the identity matrix. 
}
\end{Remark}


\subsection{Integral forms of standard modules}

Our next goal is to construct some special $R$-forms of standard modules. We call an $H_{\al,R}$-module $\De(\la)_R$ a {\em universal $R$-form of a standard module} if it is an $R$ form for both $\De(\la)_K$ and $\De(\la)_F$. We show how to construct these for all $\la$.

By Theorem~\ref{TStand}(i), for any field $k$, 
the standard module $\De(\al^m)_k$ has simple head $L(\al^m)_k$.  Pick a homogeneous generator $v\in\De(\al^m)_K$ and consider the $R$-form $\De(\al^m)_R:=H_{m\al,R}\cdot v$ of $\De(\al^m)_K$. Further, for any $\al\in Q^+$ and $\la=(\la_1^{m_1},\dots,\la_s^{m_s})\in \KP(\al)$ with $\la_1>\dots>\la_s$, we define the following $R$-form of $\De(\la)_K$ (cf. Lemma~\ref{LIndScal}): 
$$
\De(\la)_R:=\De(\la_1^{m_1})_R\circ\dots\circ \De(\la_s^{m_s})_R.
$$

Let 
$1_{(\la),R}:=1_{m_1\la_1,\dots,m_s\la_s;R}$. Then, for an appropriate set $S^{(\la)}$ of coset representatives in a symmetric group, we have that 
$
\{\tau_w1_{(\la),R}\mid w\in S^{(\la)}\}
$
is a basis of $H_{\al,R} 1_{(\la),R}$ considered as a right $1_{(\la),R}H_{\al,R}1_{(\la),R}$-module. So 
$$\De(\la)_R=\bigoplus_{w\in S^{(\la)}} \tau_w1_{(\la),R}\otimes \De(\la_1^{m_1})_R\otimes\dots\otimes \De(\la_s^{m_s})_R.
$$
In particular, choosing $v_t\in\De(\la_t^{m_t})_K$ with $\De(\la_t^{m_t})_R=H_{m_t\la_t,R}\cdot v_t$ for all $1\leq t\leq s$ and setting $v:= 1_{(\la),K}\otimes v_1\otimes\dots\otimes v_s$, we have
\begin{equation}\label{EDeRCyclic}
\De(\la)_R=H_{\al,R}\cdot v
\end{equation}

Now we show that $\De(\la)_R$ is a universal $R$-form:

\begin{Lemma} 
Let $\al\in Q^+$, and $\la\in \KP(\al)$. Then $\De(\la)_R\otimes_R F\cong \De(\la)_F$. 
\end{Lemma}
\begin{proof}
In view of  (\ref{EDelta}) and Lemma~\ref{LIndScal}, we may assume that $\la$ is of the form $(\be^m)$ for a positive root $\be$  so that $\al=m\be$. By Lemma~\ref{LRedIndep}, we have for any $\mu\in\KP(\al)$:
$$[\De(\be^m)_R\otimes_R F:L(\mu)_F]_q=\DIM\Hom_{H_{\al,K}}(P(\mu)_R\otimes_R K,\De(\be^m)_K).
$$
By convexity, $(\be^m)$ is a minimal element of $\KP(\al)$. So Lemma~\ref{LAdj} implies that all composition factors of $\De(\be^m)_R\otimes_R F$ are  $\simeq L(\be^m)_F$. Moreover,
$$
[\De(\be^m)_R\otimes_R F:L(\be^m)_F]_q=[\De(\be^m)_K:L(\be^m)_K]_q=[\De(\be^m)_F:L(\be^m)_F]_q. 
$$

By construction, $\De(\be^m)_R$ is cyclic, hence so is $\De(\be^m)_R\otimes_R F$. So, $\De(\be^m)_R\otimes_R F$ is a module with simple head 
and belongs to the category of all modules in $\mod{H_{\al,F}}$ with composition factors $\simeq L(\be^m)_F$. 
Since $(\be^m)$ is minimal in $\KP(\al)$, we have that $\De(\be^m)_F$ is the projective cover of $L(\be^m)_F$ in this category, see \cite[Lemma 4.11]{Kdonkin}. So there is a surjective homomorphism from $\De(\be^m)_F$ onto $\De(\be^m)_R\otimes_R F$. This has to be an isomorphism since we have proved that the two modules have the same composition multiplicities. 
\end{proof}

From now on, the notation $\De(\la)_R$ is reserved for a {\em universal} $R$-form.  
We begin with a rather obvious consequence of what we have proved so far:

\begin{Proposition} \label{PRTriv} 
Let $\al\in Q^+$ and $\la,\mu\in\KP(\al)$. 
\begin{enumerate}
\item[{\rm (i)}] If $\la\neq \mu$, then $\Hom_{H_{\al,R}}(\De(\la)_R,\De(\mu)_R)=0$.
\item[{\rm (ii)}] For any $R$-algebra $k$, we have $$\End_{H_{\al,R}}(\De(\la)_R)\otimes_R k = \End_{H_\al,k}(\De(\la)_R\otimes_R k).$$
\item[{\rm (iii)}] If $\la\not<\mu$, then $\Ext^j_{H_{\al,R}}(\De(\la)_R,\De(\mu)_R)=0$ for all $j\geq 1$. 
\end{enumerate}
\end{Proposition}
\begin{proof}
By the Universal Coefficient Theorem, for any $j\geq 0$, we can embed $\Ext^j_{H_{\al,R}}(\De(\la)_R,\De(\mu)_R)\otimes_R F$  into $\Ext^j_{H_{\al,F}}(\De(\la)_F,\De(\mu)_F)$. So (i) follows from  Theorem A, and (iii) follows from Theorem~\ref{TStand}(iii). Now (ii) also follows from the Universal Coefficient Theorem, since $\Ext^1_{H_{\al,R}}(\De(\la)_R,\De(\la)_R)=0$ by (iii), which makes the  $\Tor_1$-term trivial.
\end{proof}

It turns out that torsion in the Ext groups between $\De(\la)_R$'s bears some importance for Problem~\ref{Problem}, see Remark~\ref{RBears}. So we try to make progress in understanding this torsion. Given an $R$-module $V$, denote 
by $V^\Torsion$ its {\em torsion submodule}. 
If all graded components $V_d$ of $V$ are finitely generated and trivial for $d\ll 0$, then the {\em graded rank of $V$} is defined as 
$$\RANK V:=\sum_d (\rank V_d)\,q^d\in\Z((q)).$$ 
Of especial importance for us will be the rank of the torsion in Ext-groups: 
$$\RANK \Ext^j_{H_{\al,R}}(\De(\la)_R,\De(\mu)_R)^\Torsion.
$$
The following result was surprising for us: 

\begin{Theorem} \label{TTorsFree}
Let $\al\in Q^+$ and $\la,\mu\in\KP(\al)$. Then the $R$-module 
$$\Ext^1_{H_{\al,R}}(\De(\la)_R,\De(\mu)_R)$$ is torsion-free. 
\end{Theorem}
\begin{proof}
By Proposition~\ref{PRTriv}, we may assume that $\la< \mu$. 
By the Universal Coefficient Theorem, there is an exact sequence 
\begin{align*}
0
&\rightarrow \Hom_{H_{\al,R}}(\De(\la)_R,\De(\mu)_R)\otimes_R F\rightarrow \Hom_{H_{\al,F}}(\De(\la)_F,\De(\mu)_F)
\\
&\rightarrow \Tor^R_1(\Ext^{1}_{H_{\al,R}}(\De(\la)_R,\De(\mu)_R),F)\rightarrow 0.
\end{align*}
By  Theorem A, the middle term vanishes, so the third term also vanishes, which implies the theorem. 
\end{proof}

We will need the following generalization:

\begin{Corollary}\label{CTorsFreeGen}
Let $\al\in Q^+$, $\mu\in\KP(\al)$, and $V$ be an $H_{\al,R}$-module with a finite $\De$-filtration, all of whose subfactors are of the form $\simeq \De(\la)_R$ for $\la\neq \mu$. Then $\Ext^1_{H_{\al,R}}(V,\De(\mu)_R)$ is torsion-free. 
\end{Corollary}
\begin{proof}
Apply induction on the length of the $\De$-filtration, the induction base coming from Theorem~\ref{TTorsFree}. If the filtration has length greater than $1$, we have an exact sequence  $0\to V_1\to V\to V_2\to 0$, such that the inductive assumption apples to  $V_1,V_2$. Then we get a long exact sequence 
\begin{align*}
&\Hom_{H_{\al,R}}(V_1,\De(\mu)_R)\to
 \Ext^1_{H_{\al,R}}(V_2,\De(\mu)_R) 
 \\
\to
&
\Ext^1_{H_{\al,R}}(V,\De(\mu)_R)\to \Ext^1_{H_{\al,R}}(V_1,\De(\mu)_R).
\end{align*}
By  Theorem A, the first term vanishes. By the inductive assumption, the second and fourth terms are torsion-free. Hence so is the third term. 
\end{proof}

While we have just proved that there is no torsion in $\Ext^1_{H_{\al,R}}(\De(\la)_R,\De(\mu)_R)$, the following result reveals the importance of torsion in $\Ext^2$-groups.

\begin{Corollary} 
Let $\al\in Q^+$ and $\la,\mu\in\KP(\al)$. 
We have 
\begin{align*}
& \DIM \Ext^1_{H_{\al,F}}(\De(\la)_F,\De(\mu)_F)
\\
=\,\,& \DIM\Ext^1_{H_{\al,K}}(\De(\la)_K,\De(\mu)_K)+\RANK \Ext^2_{H_{\al,R}}(\De(\la)_R,\De(\mu)_R)^\Torsion.
\end{align*}
In particular, 
$$\DIM \Ext^1_{H_{\al,F}}(\De(\la)_F,\De(\mu)_F)=\DIM\Ext^1_{H_{\al,K}}(\De(\la)_K,\De(\mu)_K)$$ if and only if $\Ext^2_{H_{\al,R}}(\De(\la)_R,\De(\mu)_R)$ is torsion-free. 
\end{Corollary}
\begin{proof}
By the Universal Coefficient Theorem, there is an exact sequence 
\begin{align*}
0
&\rightarrow \Ext^1_{H_{\al,R}}(\De(\la)_R,\De(\mu)_R)\otimes_R F\rightarrow \Ext^1_{H_{\al,F}}(\De(\la)_F,\De(\mu)_F)
\\
&\rightarrow \Tor^R_1(\Ext^{2}_{H_{\al,R}}(\De(\la)_R,\De(\mu)_R),F)\rightarrow 0
\end{align*}
and an isomorphism 
$$
\Ext^1_{H_{\al,R}}(\De(\la)_R,\De(\mu)_R)\otimes_R K\cong  \Ext^1_{H_{\al,K}}(\De(\la)_K,\De(\mu)_K).
$$
The last isomorphism and Theorem~\ref{TTorsFree} imply 
$$\DIM \Ext^1_{H_{\al,K}}(\De(\la)_K,\De(\mu)_K)=\RANK \Ext^1_{H_{\al,R}}(\De(\la)_R,\De(\mu)_R).$$
On the other hand, 
$$
\RANK \Ext^{2}_{H_{\al,R}}(\De(\la)_R,\De(\mu)_R)^\Torsion=\DIM \Tor^R_1(\Ext^{2}_{H_{\al,R}}(\De(\la)_R,\De(\mu)_R),F),
$$
so the result now follows from the exactness of the first sequence. 
\end{proof}

\begin{Remark} \label{RBears} 
{\rm 
By Theorem~\ref{TTorsFree}, lack of torsion in $\Ext^2_{H_{\al,R}}(\De(\la)_R,\De(\mu)_R)$ is equivalent to the fact that the extension groups $\Ext^1_{H_{\al}}(\De(\la),\De(\mu))$ have the same graded dimension in characteristic $0$ and $p$. This is relevant for Problem~\ref{Problem}. However, we do not understand the {\em precise} connection between Problem~\ref{Problem} and lack of torsion in the groups $\Ext^2_{H_{\al,R}}(\De(\la)_R,\De(\mu)_R)$. For example, we do not know if such lack of torsion for all $\la,\mu$ implies (or is equivalent to)  James' Conjecture having positive solution. In the next section we establish a different statement of that nature. 
Set 
$$\De_k:=\bigoplus_{\la\in\KP(\al)}\De(\la)_k.$$
By the Universal Coefficient Theorem, all groups $\Ext^j_{H_{\al,R}}(\De(\la)_R,\De(\mu)_R)$ are torsion free if and only if the dimension of the $k$-algebras $\Ext^\bullet_{H_{\al,k}}(\De_k,\De_k)$ is the same for $k=K$ and $k=F$, and 
$$\Ext^\bullet_{H_{\al,k}}(\De_k,\De_k)\cong \Ext^\bullet_{H_{\al,R}}(\De_R,\De_R)\otimes_Rk$$ 
for $k=K$ and $F$. 
We do not know if James' Conjecture has positive solution under the assumption that {\em all} groups 
$\Ext^j_{H_{\al,R}}(\De(\la)_R,\De(\mu)_R)$  
are torsion-free. 
}
\end{Remark}

\subsection{Integral forms of projective modules in characteristic zero}

Recall that by lifting idempotents, we have constructed projective $R$-forms $P(\la)_R$ of the projective indecomposable modules $P(\la)_F$. 
Our next goal is to construct some interesting $R$-forms of the projective modules $P(\la)_K$. As we cannot denote them $P(\la)_R$, we will have to use the notation  $Q(\la)_R$. We will construct $Q(\la)_R$ using the usual `universal extension procedure' applied to universal $R$-forms of the standard modules, but in our `infinite dimensional integral situation' we need to be rather careful.  We begin with some lemmas. 

\begin{Lemma} \label{LPExt} 
Let $k$ be a field and $V\in\mod{{H_{\al,k}}}$ have the following properties:
\begin{enumerate}
\item[{\rm (i)}] $V$ is indecomposable;
\item[{\rm (ii)}] $V$ has a finite $\De$-filtration with the top factor $\De(\la)_k$; 
\item[{\rm (iii)}] $\Ext^1_{H_{\al,k}}(V,\De(\mu)_k)=0$ for all $\mu\in\KP(\al)$.  
\end{enumerate}
Then $V\cong P(\la)_k$.  
\end{Lemma}
\begin{proof}
We have a short exact sequence 
$
0\to M \to P\to V\to 0
$, where 
$P$ is a finite direct sum of indecomposable projective modules. 
By \cite[Corollary 7.10(i)]{Kdonkin}, $M$ has a finite $\De$-filtration. Now, by property (iii), the short exact sequence splits. Hence $V$ is projective. As it is indecomposable, it must be of the form $q^dP(\mu)$. By the property (ii), $\la=\mu$ and $d=0$.
\end{proof}

For $\la\in\KP(\al)$ and $k\in\{F,K,R\}$, we denote by $B_{\la,k}$ the endomorphism algebra $\End_{H_{\al,k}}(\De(\la)_k)^\op$. 
Then $\De(\la)_k$ is naturally a right $B_{\la,k}$-module. 
We will need to know that this $B_{\la,k}$-module is finitely generated. In fact, we will prove that it is finite rank free. First of all, this is known over a field:  


\begin{Lemma} \label{LDeBField} 
Let $\la\in\KP(\al)$ and $k$ be a field. Then:
\begin{enumerate}
\item[{\rm (i)}] $B_{\la, k}$ is a commutative polynomial algebra in finitely many variables of positive degrees. 
\item[{\rm (ii)}] Let $N_{\la,k}$ be the ideal in $B_{\la,k}$ spanned by all monomials of positive degree, and $M:=\De(\la)_kN_{\la,k}$. Then $\De(\la)_k/M\cong \bar\De(\la)_k$, see the notation (\ref{EBarDe}).
\item[{\rm (iii)}] Let $v_1,\dots,v_N\in \De(\la)_k$ be such that $\{v_1+M,\dots,v_N+M\}$ is a $k$-basis of $\De(\la)_k/M$. Then $\{v_1,\dots,v_N\}$ is a basis of $\De(\la)_k$ as a $B_{\la,k}$-module. 
\end{enumerate} 
\end{Lemma}
\begin{proof}
For (i) see Theorem~\ref{TInjGeneral}. For (ii) and (iii), see \cite[Proposition 5.7]{Kdonkin}.
\end{proof}

The following general lemma, whose proof is omitted, will help us to transfer the result of Lemma~\ref{LDeBField} from $K$ and $F$ to $R$:

\begin{Lemma} \label{LBasisRFK} 
Let $B_R$ be an $R$-algebra and $V_R$ be a $B_R$-module. Assume that $B_R$ and $V_R$ are free as $R$-modules. If $v_1,\dots,v_N\in V_R$ are such that $\{v_1\otimes 1_k,\dots, v_N\otimes 1_k\}$ is a basis of $V_R\otimes_R k$ as a $B_R\otimes_Rk$-module for $k=K$ and $F$, then $\{v_1,\dots, v_N\}$ is a basis of $V_R$ as a $B_R$-module. 
\end{Lemma}

\begin{Lemma} \label{LDeB} 
Let $\la\in\KP(\al)$. As a $B_{\la,R}$-module, $\De(\la)_R$ is finite rank free. 
\end{Lemma}
\begin{proof}
Let $\la=(\la_1^{m_1},\dots,\la_s^{m_s})$ for positive roots $\la_1>\dots>\la_s$. Choose $v=1_{(\la),K}\otimes v_1\otimes\dots\otimes v_s$ as in (\ref{EDeRCyclic}). There is a submodule $M\subset \De(\la)_K$ with $\De(\la)_K/M\cong \bar\De(\la)_K$. Pick $h_1,\dots,h_N\in H_{\al,R}$ such that $\{h_1v+M,\dots,h_Nv+M\}$ is an $R$-basis of $\bar\De(\la)_R=H_{\al,R}\cdot (v+M)$. By Lemma~\ref{LDeBField}, $\{h_1v\otimes 1_k,\dots,h_Nv\otimes 1_k\}$ is a $B_{\la,k}$-basis of $\De(\la)_R\otimes_R k$ for $k=K$ or $F$. Now apply Proposition~\ref{PRTriv}(ii) and Lemma~\ref{LBasisRFK}.
\end{proof}

\begin{Corollary} \label{CExtFG}
Let $k\in\{F,K,R\}$, $V\in\mod{H_{\al,k}}$, $\la\in\KP(\al)$ and $j\in\Z_{\geq 0}$. Then $\Ext^j_{H_{\al,k}}(V,\De(\la)_k)$ is finitely generated as a $B_{\la,k}$-module.  
\end{Corollary}
\begin{proof}
Since $H_{\al,k}$ is Noetherian, $V$ has a resolution by finite rank free modules over $H_{\al,k}$. Applying $\Hom_{H_{\al,k}}(-,\De(\la)_k)$ to this resolution, we get a complex with terms being finite direct sums of modules $\simeq \De(\la)_k$, which are finite rank free over $B_{\la,k}$, thanks to Lemmas~\ref{LDeBField} and \ref{LDeB}. As $B_{\la,k}$ is Noetherian, the cohomology groups of the complex are finitely generated $B_{\la,k}$-modules. 
\end{proof}

\begin{Remark} 
{\rm 
It is a more subtle issue to determine whether $\Ext^j_{H_{\al,k}}(\De(\la)_k,V)$ is finitely generated as a $B_{\la,k}$-module. We do not know if this is always true.  
}
\end{Remark}

\begin{Lemma} \label{LUEP} 
{\bf (Universal Extension Procedure)}
Let $k\in\{F,K,R\}$, $\mu\in\KP(\al)$, and $V_k$ be an indecomposable $H_{\al,k}$-module with a finite $\De$-filtration, all of whose subfactors are of the form $\simeq \De(\la)_k$ for $\la\not\geq \mu$. If $k=R$, assume in addition that $V_R\otimes_R K$ is indecomposable. 
Let   
$$r(q):=\RANK \Ext^1_{H_{\al,k}}(V_k,\De(\mu)_k)\in\Z[q,q^{-1}]$$
be the rank of $\Ext^1_{H_{\al,k}}(V_k,\De(\mu)_k)$ as a $B_{\mu,k}$-module. Then there exists an $H_{\al,k}$-module $E(V_k,\De(\mu)_k)$ with the following properties:
\begin{enumerate}
\item[{\rm (i)}] $E(V_k,\De(\mu)_k)$ is indecomposable;
\item[{\rm (ii)}] $\Ext^1_{H_{\al,k}}(V_k,\De(\mu)_k)=0$;
\item[{\rm (iii)}] there is a short exact sequence\,
$$
0\to \overline{r(q)}\De(\mu)_k\to E(V_k,\De(\mu)_k)\to V_k\to 0.
$$
\end{enumerate}
\end{Lemma}
\begin{proof}
In this proof we drop $H_{\al,k}$ from the indices and write $\Ext^1$ for $\Ext^1_{H_{\al,k}}$, etc. Also, when this does not cause a confusion, we drop $k$ from the indices. 
Let $\xi_1,\dots,\xi_r$ be a minimal set of homogeneous generators of $\Ext^1(V,\De(\mu))$ as a $B_{\mu}$-module, and $d_s:=\deg(\xi_s)$ for $s=1,\dots,r$, so that $r(q)=\sum_sq^{d_s}$. 
The extension 
$
0\to q^{-d_1}\De(\mu)\to E_1\to V\to 0, 
$ 
 corresponding to $\xi_1$, yields the long exact sequence
\begin{align*} \HOM(q^{-d_1}\De(\mu),\De(\mu))
\stackrel{\phi}{\longrightarrow} \Ext^1(V,\De(\mu))
\stackrel{\psi}{\longrightarrow}
\Ext^1(E_1,\De(\mu))\to 0.
\end{align*}
We have used that $\EXT^1(q^{-d_1}\De(\mu),\De(\mu))=0$, see Proposition~\ref{PRTriv}(iii). 
Note that $q^{-d_1}\De(\mu)=\De(\mu)$ as $H_\al$-modules but with degrees shifted down by $d_1$. So we can consider the identity map $\id: q^{-d_1}\De(\mu)\to \De(\mu)$, which has degree $d_1$. 
The connecting homomorphism $\phi$ maps this  identity map to $\xi_1$. It follows that $\Ext^1(E_1,\De(\mu))$ is generated as a $B_{\mu}$-module by the elements $\bar\xi_2:=\psi(\xi_2),\dots,\bar\xi_r:=\psi(\xi_r)$. 
Repeating the argument $r-1$ more times, 
we get an extension 
$$
0\to q^{-d_1} \De(\mu)\oplus\dots\oplus q^{-d_r} \De(\mu)=\overline{r(q)}\De(\mu)
\to E\to V\to 0
$$ 
such that in the corresponding long exact sequence 
\begin{align*}
&\HOM(E,\De(\mu)) \stackrel{\chi}{\longrightarrow} \HOM(\overline{r(q)}\De(\mu),\De(\mu))
\\
\stackrel{\phi}{\longrightarrow} &\Ext^1(V,\De(\mu))
\to
\Ext^1(E,\De(\mu))\to 0,
\end{align*}
for all $s=1,\dots,r$, we have $\phi(\pi_s)=\xi_s$, where $\pi_s$ is the (degree $d_s$) projection onto the $s$th summand. In particular, $\phi$ is surjective, 
and $\Ext^1(E,\De(\mu))=0.$ 

It remains to prove that $E$ is indecomposable. We first prove this when $k$ is a field. In that case, if $E=E'\oplus E''$, then both $E'$ and $E''$ have finite $\De$-filtrations, see \cite[Corollary 7.10]{Kdonkin}. Since $\Ext^1(\De(\mu),\De(\la))=0$ for $\la\not>\mu$, there is a partition $J'\sqcup J''=\{1,\dots,r\}$ such that there are submodules  $M'\cong\oplus_{j\in J'}q^{d_j}\De(\mu)\subseteq E'$, $M''\cong\oplus_{j\in J''}q^{d_j}\De(\mu)\subseteq E''$  and  $E'/M'$, $M''/E''$ have $\De$-filtrations. Since $\Hom(\De(\mu),V)=0$, we now deduce that $V\cong E'/M'\oplus E''/M''$. As $V$ is indecomposable, we may assume that $E'/M'=0$. Then some projection $\pi_s$ lifts to a homomorphism $E\to \De(\mu)$, which shows that this $\pi_s$ is in the image of $\chi$, and hence in the kernel of $\phi$, which is a contradiction. 

Now let $k=R$. Note that $V$ and $E$ are free as $R$-modules since so are all $\De(\nu)_R$'s. If $E_R$ is decomposable, then so is $E_R\otimes K$, so it suffices to prove that $E_R\otimes K$ is indecomposable. In view of Corollary~\ref{CTorsFreeGen}, the $B_{\mu,K}$-module $\Ext^1(V_R,\De(\mu)_R)\otimes_R K \cong \Ext^1(V_R\otimes_R K,\De(\mu)_K)$ is minimally generated by $\xi_{1,R}\otimes 1_K,\dots,\xi_{r,R}\otimes 1_K$. It follows, using Lemma~\ref{LExtExt},  that $E_R\otimes_R K\cong E_K$, where $E_K$ is constructed using the universal extension procedure starting with the indecomposable module $V_K:=V_R\otimes_R K$ as in the first part of the proof of the lemma. By the field case established in the previous paragraph, $E_K$ is indecomposable. 
\end{proof}

Let $\la\in\KP(\al)$. For  $k\in\{R,K,F\}$, we construct a module $Q(\la)_k$ starting  with $\De(\la)_k$, and  repeatedly applying the universal extension procedure. To simplify notation we drop some of the indices $k$ if this does not lead to a confusion. Given Laurent polynomials $r_0(q),r_1(q),\dots,r_m(q)\in\Z[q,q^{-1}]$ with non-negative coefficients and Kostant partitions $\la^0,\la^1,\dots,\la^m\in\KP(\al)$, we will use the notation 
$$V=r_0(q)\De(\la^0)\mid r_1(q)\De(\la^1)\mid \dots\mid r_m(q)\De(\la^m)
$$
to indicate that the $H_\al$-module $V$ has a filtration $V=V_0\supseteq V_1\supseteq \dots\supseteq V_{m+1}=(0)$ such that $V_s/V_{s+1}\cong r_s(q)\De(\la^s)$ for $s=0,1\dots,m$. 

If 
$\Ext^1_{H_{\al}}(\De(\la),\De(\mu))=0$ 
for all $\mu\in\KP(\al)$, we  set $Q(\la)_k:=\De(\la)_k$.  
Otherwise, let $\la^{1,k}\in\KP(\al)$ be minimal with 
$\Ext^1_{H_{\al}}(\De(\la),\De(\la^{1,k}))\neq0.$  
Note that this $\la^{1,k}$ might indeed depend on the ground ring  $k$, hence the notation. Also notice $\la^{1,k}>\la$. Let $E(\la,\la^{1,k})_k:=E(\De(\la),\De(\la^{1,k}))$, see Lemma~\ref{LUEP}. By construction $E(\la,\la^{1,k})_k=\De(\la)\mid  \overline{r_{1,k}(q)}\De(\la^{1,k})$, where 
$$r_{1,k}(q)=\RANK \Ext^1_{H_{\al}}(\De(\la),\De(\la^{1,k}))$$ as a $B_{\la^{1,k}}$-module. This rank might depend on $k$, hence the notation. If 
$$\Ext^1_{H_{\al}}(E(\la,\la^{1,k}),\De(\mu))=0$$ 
for all $\mu\in\KP(\al)$, we set $Q(\la)_k:=E(\la,\la^{1,k})_k$.  
Otherwise, let $\la^{2,k}\in\KP(\al)$ be minimal with 
$\Ext^1_{H_{\al}}(E(\la,\la^{1,k}),\De(\la^{2,k}))\neq0.
$ 
Note that $\la^{2,k}>\la$ and $\la^{2,k}\neq \la^{1,k}$. 
Let $E(\la,\la^{1,k},\la^{2,k})_k:=E(E(\la,\la^{1,k}),\De(\la^{2,k}))$. By construction 
$$E(\la,\la^{1,k},\la^{2,k})_k=\De(\la)\mid \overline{r_{1,k}(q)}\De(\la^{1,k})\mid  \overline{r_{2,k}(q)}\De(\la^{2,k}),
$$ 
where 
$$r_{2,k}(q)=\RANK \Ext^1_{H_{\al}}(E(\la,\la^{1,k}),\De(\la^{2,k}))$$ as a $B_{\la^{2,k}}$-module. 
If 
$\Ext^1_{H_{\al}}(E(\la,\la^{1,k},\la^{2,k}),\De(\mu))=0$  
for all $\mu\in\KP(\al)$, we set $Q(\la)_k:=E(\la,\la^{1,k},\la^{2,k})$.  

Since on each step we will have to pick $\la^{t,k}>\la$, which does not belong to $\{\la,\la^{1,k},\dots,\la^{t-1,k}\}$, the process will stop after finitely many steps, and we will obtain a module $$E(\la,\la^{1,k},\dots,\la^{m_k,k})_k=\De(\la)\mid \overline{r_{1,k}(q)}\De(\la^{1,k})\mid  \dots\mid  \overline{r_{m_k,k}(q)}\De(\la^{m_k,k}),$$ where 
\begin{equation}\label{ERankUn}
r_{t,k}(q)=\RANK \Ext^1_{H_{\al,k}}(E(\la,\la^{1,k},\dots,\la^{t-1,k})_k,\De(\la^{t,k})_k) 
\end{equation}
 as a $B_{\la^{t,k},k}$-module for all $1\leq t\leq m_k$, and such that 
$$\Ext^1_{H_{\al,k}}(E(\la,\la^{1,k},\dots,\la^{m_k,k})_k,\De(\mu)_k)=0$$ 
for all $\mu\in\KP(\al)$. We set $Q(\la)_k:=E(\la,\la^{1,k},\dots,\la^{m_k,k})_k$. 

\begin{Theorem} \label{TQ}
Let $\al\in Q^+$ and $\la\in \KP(\al)$. 
\begin{enumerate}
\item[{\rm (i)}] For $k=K$ or $F$, we have $Q(\la)_k\cong P(\la)_k$. 
\item[{\rm (ii)}] For $k=K$ or $F$, the rank $r_{t,k}(q)$ from (\ref{ERankUn}) equals the decomposition number $d_{\la^{t,k},\la}^k$ for all $1\leq t\leq m_k$, and $d_{\mu,\la}^k=0$ for  $\mu\not\in \{\la^{t,k}\mid 1\leq t\leq m_k\}$. 

\item[{\rm (iii)}] $m_R=m_K$; setting $m:=m_R$, we may choose $\la^{1,R}=\la^{1,K},\dots,\la^{m,R}=\la^{m,K}$ and then  $r_{t,R}(q)=r_{t,K}(q)$ for all $1\leq t\leq m$. 

\item[{\rm (iv)}] $Q(\la)_R\otimes_R K\cong P(\la)_K$. 

\end{enumerate}
\end{Theorem}
\begin{proof}
Part (i) follows from the construction and Lemma~\ref{LPExt}. Part (ii) follows from part (i), the construction, and Theorem~\ref{TStand}(v). 

To show (iii) and (iv), we prove by induction on $t=0,1,\dots$ that we can choose $\la^{t,R}=\la^{t,K}$, $r_{t,R}(q)=r_{t,K}(q)$ and 
\begin{equation}\label{EEExt}
E(\la,\la^{1,R},\dots,\la^{t,R})_R\otimes_R K\cong E(\la,\la^{1,K},\dots,\la^{t,K})_K.
\end{equation}
The induction base is simply the statement $\De(\la)_R\otimes_RK\cong \De(\la)_K$. 
For the induction step, assume that $t>0$ and the claim has been proved for all $s<t$. 

Let $\xi_{1,R},\dots,\xi_{r,R}$ be a minimal set of generators of the $B_{\la^{t,R},R}$-module 
$$\Ext^1_{H_{\al,R}}(E(\la,\la^{1,R},\dots,\la^{t-1,R})_R,\De(\la^{t,R})_R),$$ 
so that $r_{t,R}(q)=\deg(\xi_{1,R})+\dots+\deg(\xi_{r,R})$ is the rank of that module. 
Using Corollary~\ref{CTorsFreeGen} and the Universal Coefficient Theorem, we deduce that $\la^{t,K}$ can be chosen to be $\la^{t,R}$ and the $B_{\la^{t,R},K}$-module 
$$\Ext^1(\De(\la)_R,\De(\la^{t,R})_R)\otimes_R K \cong \Ext^1(V_R\otimes_R K,\De(\la^{t,R})_K)$$ 
is minimally generated by $\xi_{1,R}\otimes 1_K,\dots,\xi_{r,R}\otimes 1_K$, so that $r_{t,K}(q)=r_{t,R}(q)$. Finally (\ref{EEExt}) comes from Lemma~\ref{LExtExt}. 
\end{proof}

In view of Theorem~\ref{TQ}(i), $Q(\la)_R$ in general is not an $R$-form of $Q(\la)_F\cong P(\la)_F$. 
For every $\la\in\KP(\al)$, define the $H_{\al,F}$-module $X(\la):=Q(\la)_R\otimes F$.

\begin{Theorem} 
James' Conjecture has positive solution for $\al$ if and only if one of the following equivalent conditions holds:
\begin{enumerate}
\item[{\rm (i)}] $X(\la)$ is projective;
\item[{\rm (ii)}] $X(\la)\cong P(\la)_F$ for all $\la\in\KP(\al)$;
\item[{\rm (iii)}] $\Ext^1_{H_{\al,F}}(X(\la),\De(\mu)_F)=0$ for all $\la,\mu\in\KP(\al)$;
\item[{\rm (iv)}] $\Ext^2_{H_{\al,R}}(Q(\la)_R,\De(\mu)_R)$ is torsion-free for all $\la,\mu\in\KP(\al)$. 
\end{enumerate}
\end{Theorem}
\begin{proof}
(i) and (ii) are equivalent by an argument involving formal characters and Lemma~\ref{LAdj}. 
Furthermore, (i) and (iii) are equivalent by Lemma~\ref{LPExt}. Since since $\Ext^1_{H_{\al,R}}(Q(\la)_R,\De(\mu)_R)=0$ for all $\mu$, (iii) is equivalent to (vi) by the Universal Coefficient Theorem. Finally,  we prove that (ii) is equivalent to James' Conjecture having   positive solution. If $X(\la)\cong P(\la)_F$ for all $\la$, then they have the same graded dimension, so the $R$-modules $Q(\la)_R$ and $P(\la)_R$ have  the same graded $R$-rank, whence the $K$-modules $P(\la)_K\cong Q(\la)_R\otimes_R K$ and $P(\la)_R\otimes_R K$ have the same graded dimension, therefore $P(\la)_R\otimes_R K\cong P(\la)_K$ for all $\la$, see Lemma~\ref{LAdj}, whence James' Conjecture has  positive solution. 

Conversely, assume James' Conjecture has positive solution. This means that $d_{\mu,\la}^K=d_{\mu,\la}^F$ for all $\mu,\la\in\KP(\al)$. By Theorem~\ref{TQ}(ii), on every step of our universal extension process, we are going to have the same rank of the $\Ext^1$-group over $K$ and $F$, so, by Theorem~\ref{TQ}(iii), on every step of our universal extension process, we are also going to have the same rank of the appropriate $\Ext^1$-groups over $R$ and $F$. Now, use Lemma~\ref{LExtExt} as in the proof of Theorem~\ref{TQ}(iv) to show that $Q(\la)_R\otimes_R F\cong P(\la)_F$. 
\end{proof}

\begin{Remark} 
{\rm 
We conjecture that $P(\la)_F$ has an $X$-filtration with the top quotient $X(\la)$ and $X(\mu)$ appearing $a_{\mu,\la}(q)$ times. On the level of Grothendieck groups, this is true thanks to Lemma~\ref{LAdj}. But it seems not so obvious even that $X(\la)$ is a quotient of $P(\la)_F$. 
}
\end{Remark}


\end{document}